\documentclass{amsart}
\usepackage[utf8]{inputenc}
\usepackage{amssymb, amsfonts, amsmath,amsthm}
\usepackage{mathrsfs}
\usepackage{hyperref}
\usepackage{graphicx}
\usepackage{caption}
\usepackage{subcaption}
\usepackage{wrapfig}
\usepackage[margin=1.00 in]{geometry}
\usepackage{enumitem, multicol}
\usepackage[dvipsnames]{xcolor}
\usepackage{tikz}
\usepackage{nicefrac, bigints}
\usetikzlibrary{decorations.markings}
\usepackage{comment}
\usepackage{thmtools} 
\usepackage{thm-restate}
\usepackage{hyperref}
\usepackage{pgfplots}
\usepgfplotslibrary{fillbetween}
\pgfplotsset{compat=1.15}

\newtheorem{lemma}{Lemma}
\newtheorem{prop}{Proposition}

\newtheorem{remark}{Remark}

\newtheorem{defn}{Definition}

\newcommand{\ZZ}{\mathbb{Z}}
\newcommand{\RR}{\mathbb{R}}

\newcommand{\CC}{\mathbb{C}}

\newcommand{\scrL}{\mathscr{L}}

\tikzset{-|-/.style={decoration={
  markings,
  mark=at position .5 with {\arrow{|}}},postaction={decorate}}}

\tikzset{-||-/.style={decoration={
  markings,
  mark=at position .47 with {\arrow{|}}, 
  mark=at position .53 with {\arrow{|}}},postaction={decorate}}}

\tikzset{-|||-/.style={decoration={
  markings,
  mark=at position .44 with {\arrow{|}}, 
  mark=at position .5 with {\arrow{|}},
  mark=at position .56 with {\arrow{|}}},postaction={decorate}}}

\tikzset{-||||-/.style={decoration={
  markings,
  mark=at position .41 with {\arrow{|}},
  mark=at position .47 with {\arrow{|}}, 
  mark=at position .53 with {\arrow{|}},
  mark=at position .59 with {\arrow{|}}},postaction={decorate}}}

\tikzset{ma/.style={
        decoration={markings,
            mark= at position 0.5 with {\arrow{#1}} ,
        },
        postaction={decorate}
    }
}

\begin{document}

\title{Slope Gap Distributions of Veech surfaces
}


\author{Luis Kumanduri        \and
        Anthony Sanchez                        \and
        Jane Wang
}




\maketitle

\begin{abstract} 
The slope gap distribution of a translation surface is a measure of how random the directions of the saddle connections on the surface are. It is known that Veech surfaces, a highly symmetric type of translation surface, have gap distributions that are piecewise real analytic.  Beyond that, however, very little is currently known about the general behavior of the slope gap distribution, including the number of points of non-analyticity or the tail. 

We show that the limiting gap distribution of slopes of saddle connections on a Veech translation surface is always piecewise real-analytic with \emph{finitely} many points of non-analyticity. We do so by taking an explicit parameterization of a Poincar\'{e} section to the horocycle flow on $\text{SL}(2,\mathbb{R})/\text{SL}(X,\omega)$ associated to an arbitrary Veech surface $\text{SL}(X,\omega)$ and establishing a key finiteness result for the first return map under this flow.  We use the finiteness result to show that the tail of the slope gap distribution of Veech surfaces always has quadratic decay.
\end{abstract}

\section{Introduction} 
\label{sec:intro}

In this paper, we will study the slope gap distributions of Veech surfaces, a highly symmetric type of translation surface. \textbf{Translation surfaces} can be defined geometrically as finite collections of polygons with sides identified in parallel opposite pairs. If we place these polygons in the complex plane $\CC$, the surface inherits a Riemann surface structure from $\CC$, and the one form $dz$ gives rise to a well-defined holomorphic one-form on the surface. This leads to a second equivalent definition of a translation surfaces as a pair $(X, \omega)$ where $X$ is a Riemann surface and $\omega$ is a holomorphic one-form on the surface. Every translation surface locally has the structure of $(\CC,dz)$, except for at finitely many points that have total angle around them $2\pi k$ for some integer $k \geq 2$. These points are called \textbf{cone points} and correspond to the zeros of the one-form $\omega$. A zero of order $n$ gives rise to a cone point of angle $2\pi (k+1)$. 

A translation surface inherits a flat metric from $\CC$. \textbf{Saddle connections} are then straight line geodesics connecting two cone points with no cone points in the interior. The \textbf{holonomy vector} of a saddle connection $\gamma$ is then the vector describing how far and in what direction the saddle connection travels, 
$$\mathbf{v}_\gamma = \int_\gamma \omega .$$
We will be interested in the distribution of directions of these vectors for various translation surfaces. 

There is a natural $\text{SL}(2,\mathbb{R})$ action on translation surfaces coming from the linear action of matrices on $\mathbb{R}^2$, as can be seen in Figure \ref{fig:sl2r_action}.

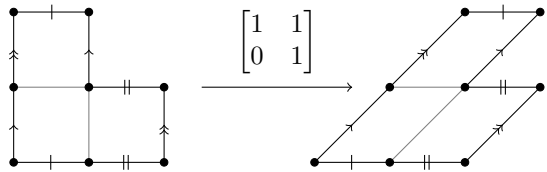
\begin{figure}[ht]
\centering
\begin{tikzpicture}
\draw[-|-] (0,0) -- (1,0); 
\draw[-||-] (1,0) -- (2,0); 
\draw[ma={>>}] (2,0) -- (2,1); 
\draw[-||-] (2,1) -- (1,1); 
\draw[ma={>}] (1,1) -- (1,2); 
\draw[-|-] (1,2) -- (0,2); 
\draw[ma={>}] (0,0) -- (0,1); 
\draw[ma = {>>}] (0,1) -- (0,2); 
\draw[gray] (1,0) -- (1,1) -- (0,1);
\draw[fill] (0,0) circle [radius = 0.05];
\draw[fill] (1,0) circle [radius = 0.05];
\draw[fill] (2,0) circle [radius = 0.05];
\draw[fill] (2,1) circle [radius = 0.05];
\draw[fill] (1,1) circle [radius = 0.05];
\draw[fill] (1,2) circle [radius = 0.05];
\draw[fill] (0,2) circle [radius = 0.05];
\draw[fill] (0,1) circle [radius = 0.05];

\node at (3.5,1.6) {$\begin{bmatrix} 1 & 1 \\ 0 & 1 \end{bmatrix}$};
\draw[->] (2.5, 1) -- (4.5,1);

\draw[-|-] (4,0) -- (5,0); 
\draw[-||-] (5,0) -- (6,0); 
\draw[ma={>>}] (6,0) -- (7,1); 
\draw[-||-] (7,1) -- (6,1); 
\draw[ma={>}] (6,1) -- (7,2); 
\draw[-|-] (7,2) -- (6,2); 
\draw[ma={>}] (4,0) -- (5,1); 
\draw[ma = {>>}] (5,1) -- (6,2); 
\draw[gray] (5,0) -- (6,1) -- (5,1);
\draw[fill] (4,0) circle [radius = 0.05];
\draw[fill] (5,0) circle [radius = 0.05];
\draw[fill] (6,0) circle [radius = 0.05];
\draw[fill] (7,1) circle [radius = 0.05];
\draw[fill] (6,1) circle [radius = 0.05];
\draw[fill] (7,2) circle [radius = 0.05];
\draw[fill] (6,2) circle [radius = 0.05];
\draw[fill] (5,1) circle [radius = 0.05];
\end{tikzpicture} 
\caption{A matrix in $\text{SL}(2,\RR)$ acting on a translation surface.}
\label{fig:sl2r_action}
\end{figure}

Sometimes this action produces a symmetry of the surface $(X, \omega)$. That is, after acting on the surface by the matrix, it is possible to cut and past the new surface so that it looks like the original surface again. The collection of these symmetries is the stabilizer under the $\text{SL}(2,\mathbb{R})$ action and is called the \textbf{Veech group} of the surface. It will be denoted $\text{SL}(X, \omega)$ and is a subgroup of $\text{SL}(2,\mathbb{R})$. When the Veech group $\text{SL}(X, \omega)$ of a translation surface has finite covolume in $\text{SL}(2,\mathbb{R})$, the surface $(X, \omega)$ is called a \textbf{Veech surface}. Sometimes such surfaces are also called lattice surfaces since $\text{SL}(X, \omega)$ is a lattice in $\text{SL}(X,\mathbb{R})$. Veech surfaces have many nice properties, such as satisfying the \emph{Veech dichotomy}: in any direction, every infinite trajectory on the surface is periodic or every infinite trajectory is equidistributed. For more information about translation and Veech surfaces, we refer the reader to \cite{HS} and \cite{Z}.

From work of Vorobets (\cite{V}), it is known that for almost every translation surface $(X,\omega)$ with respect to the Masur-Veech volume on any strata of translation surfaces (for details about Masur-Veech volume and strata, please see \cite{Z}), the angles of the saddle connections equidistribute in $S^1$. That is, if we let $$\Lambda(X, \omega) := \{\text{holonomy vectors of saddle connections of } (X, \omega)\},$$ and we normalize the circle to have total length $1$, then for any interval $I \subset S^1$, as we let $R \rightarrow \infty$, the proportion of vectors in $\Lambda(X,\omega)$ of length $\leq R$ that have direction in the interval $I$ converges to the length of $I$. 

A finer measure of the randomness of the saddle connection directions of a surface is its \textbf{gap distribution}, which we will now define. The idea of the gap distribution is it records the limiting distribution of the spacings between the set of angles (or in our case, slopes) of the saddle connection directions of length up to a certain length $R$. We will be working with slope gap distributions rather than angle gap distributions because the slope gap distribution has deep ties to the horocycle flow on strata of translation surfaces. Thus, dynamical tools relating to the horocycle flow can be more easily applied to analyze the slope gap distribution. 

Let us restrict our attention to the first quadrant and to slopes of at most $1$ and define 
$$\mathbb{S}(X, \omega) := \{\text{slope} (\mathbf{v}) : \mathbf{v} \in \Lambda(X, \omega) \text{ and } 0 < \text{Re}(\mathbf{v}), 0 \leq \text{Im}(\mathbf{v}) \leq \text{Re}(\mathbf{v})\}.$$

We also allow ourselves to restrict to slopes of saddle connections of at most some length $R$ in the $\ell_\infty$ metric, and define $$\mathbb{S}_R(X, \omega) := \{\text{slope} (\mathbf{v}) : \mathbf{v} \in \Lambda(X, \omega) \text{ and } 0 < \text{Re}(\mathbf{v}), 0 \leq \text{Im}(\mathbf{v}) \leq \text{Re}(\mathbf{v})\leq R\}.$$

We let $N(R)$ denote the number of unique slopes $N(R) := |\mathbb{S}_R(X, \omega)|$. By results of Masur (\cite{M}, \cite{M2}), the growth of the number of saddle connections of length at most $R$ in any translation surface is quadratic in $R$. We can order the slopes: 
$$\mathbb{S}_R (X, \omega) = \{0 \leq s_0^R < s_1^R < \cdots < s_{N(R)-1}^R\}.$$ 
Since $N(R)$ grows quadratically in $R$, we now define the \textbf{renormalized slope gaps} of $(X, \omega)$ to be $$\mathbb{G}_R(X, \omega) := \{R^2 (s_i^R - s_{i-1}^R) : 1 \leq i \leq N(R) - 1 \text{ and } s_i \in \mathbb{S}_R(X,\omega)\}.$$ 

If there exists a limiting probability distribution function $f : [0,\infty) \rightarrow [0,\infty)$ for the renormalized slope gaps $$\lim_{R \rightarrow \infty} \frac{|\mathbb{G}_R(X, \omega) \cap (a,b)|}{N(R)} = \int_a^b f(x) \, dx,$$ then $f$ is called the \textbf{slope gap distribution} of the translation surface $(X, \omega)$. If the sequence of slopes of holonomy vectors of increasing length of a translation surface were independent and identically distributed uniform $[0,1]$ random variables, then a probability computation shows that the gap distribution would be a Poisson process of intensity $1$. In all computed examples of slope gap distributions, however, this is not the case.

We now give a brief overview of the literature gap distributions of translation surfaces. In \cite{AC1}, Athreya and Chaika analyzed the gap distributions for typical surfaces and showed that for almost every translation surface (with respect to the Masur-Veech volume), the gap distribution exists. They also showed that a translation surface is a Veech surface if and only if it has \emph{no small gaps}, that is, if $\liminf_{R \rightarrow \infty} (\min(\mathbb{G}_R(X,\omega)) > 0$. In a later work \cite{ACL}, Athreya, Chaika, and Leli\'{e}vre explicitly computed the gap distribution of the golden L and in \cite{A}, Athreya gives an overview of results and techniques about gap distributions. Another relevant work is a paper by Taha (\cite{Taha}) studying cross sections to the horocycle and geodesic flows on quotients of $\text{SL}(2, \mathbb{R})$ by Hecke triangle groups. The computation of slope gap distributions involved understanding the first return map of the horocycle flow to a particular transversal of a quotient of $\text{SL}(2, \mathbb{R})$.  

In \cite{UW}, Uyanik and Work computed the gap distribution of the octagon, and also showed that the gap distribution of any Veech surface exists and is piecewise real analytic. The second named author then went on to study the gap distributions of doubled slit tori (\cite{S}). Up until then, all known slope gap distributions were for Veech surfaces. The above articles focus on gap distributions of the $\text{SL}(2,\mathbb R)$-orbit of \emph{specific} translation surfaces. This work applies to any Veech surface and gives insight to the general behavior of the graph of the slope gap distribution of Veech surfaces. In fact, outside of [AC12] where they show there are no small gaps, there is no other works in this direction with this level of generality. 

Uyanik and Work gave an algorithm to compute the gap distribution of any Veech surface and showed that the gap distribution was piecewise analytic. However, their algorithm does not necessarily terminate in finite time and can make it seem like the gap distribution can have infinitely many points of non-analyticity, as we will see in Section \ref{subsec:examples}. We improve upon their algorithm to guarantee termination in finite time and show as a result that every Veech translation surface has a gap distribution with finitely many points of non-analyticity. Uyanik and Work's algorithm starts by taking a tranversal to the horocycle flow which a priori may break up into infinitely many components under the return map. Our key observation is that by carefully choosing this transversal using the geometry of our surface, we see that it will only break up into finitely many pieces, which will give the following theorem.

\begin{restatable*}{thm}{finite} 
\label{thm:finite}

The slope gap distribution of any Veech surface has finitely many points of non-analyticity.
\end{restatable*}

In addition, we show that the tail of the gap distribution of any Veech surface has a quadratic decay. Let $f(t)\sim g(t)$ mean that the ratio is bounded above and below by two positive constants. 

\begin{restatable*}{thm}{tail} 
\label{thm:tail}
The slope gap distribution of any Veech surface has quadratic tail decay. That is, if $f$ denotes the density function of the slope gap distribution, then 
$$ \int_t ^\infty f(x)\,dx\sim{t^{-2}}.$$
\end{restatable*} 

Thus, the results of this paper and the ``no small gaps" result of \cite{AC1} give a us a good understanding of the graph of the slope gap distribution of Veech surfaces: for some time the graph is identically zero before becoming positive. Afterward the graph has finitely many pieces where it is real analytic and fluctuates up and down before it begins permanently decaying quadratically.

\textbf{Organization of the paper.} In Section \ref{subsec:poincare} we will go over background information on slope gap distributions, including how to relate the gap distribution to return times to a Poincar\'{e} section of the horocycle flow. In Section \ref{subsec:algorithm}, we will outline the algorithm of Uyanik and Work, and observe some possible modifications. In Section \ref{subsec:examples}, we will see how a couple steps of Uyanik and Work's algorithm apply to a specific Veech surface. A priori, the first return map to the Poincar\'{e} section breaks the section into infinitely many pieces, but after making some modifications to the parameterization, we will see that there are in fact finitely many pieces. In Section \ref{sec:main} we will give a proof of Theorem \ref{thm:finite}. The strategy of the proof is to apply a compactness argument to show finiteness under our modified parameterization of the Poincar\'{e} section. We will show that on a compact set that includes the Poincar\'{e} section, every point has a neighborhood that can contribute at most finitely many points of non-analyticity to the gap distribution. This will give us that the slope gap distribution has finitely many points of non-analyticity overall. In Section \ref{sec:decay}, as an application of Theorem \ref{thm:finite}, we prove quadratic decay of the slope gap distribution of Veech surfaces. Finally, in Section \ref{sec:questions}, we discuss a few further questions regarding slope gap distributions of translation surfaces.

\textbf{Acknowledgements.} The first and third authors would like to thank Moon Duchin for organizing the Polygonal Billiards Research Cluster held at Tufts University in 2017, where this work began, as well as all of the participants of the cluster. We would also like to thank Jayadev Athreya, Aaron Calderon, Jon Chaika, Samuel Leli\'{e}vre, Caglar Uyanik, and Grace Work for helpful conversations about limiting gap distributions. This work was supported by the National Science Foundation under grant number DMSCAREER-1255442, by the National Science Foundation Graduate Research Fellowship under grant numbers 1745302 (LK) and 1122374 (JW), and the National Science Foundation Postdoctoral Fellowship under grant number DMS-2103136 (AS).


\section{Background}
\label{sec:background} 

\subsection{A Poincar\'{e} section for the horocycle flow}
\label{subsec:poincare}
In this section, we review a general strategy for computing the gap distribution of a translation surface by relating slope gap distributions to the horocycle flow. For more background and proofs of the statements given here, see \cite{AC2} or \cite{ACL}. 

Suppose that we wish to compute the slope gap distribution of a translation surface $(X, \omega)$. We let $\Lambda(X, \omega)$, sometimes shortened to just $\Lambda$, be the set of holonomy vectors of the surface. We may start by considering all of the holonomy vectors of $(X, \omega)$ in the first quadrant, with $\ell_\infty$ norm $\leq R$. If we act on $(X,\omega)$ by the matrix $$g_{-2 \log(R)} = \begin{bmatrix} 1/R & 0 \\ 0 & R \end{bmatrix},$$ the slopes of the holonomy vectors of $g_{-2\log(R)} (X,\omega)$ in $[0,1]\times [0,R^2]$ is the same as $R^2$ times the slopes of the holonomy vectors of $(X,\omega)$ in $[0,R] \times [0,R]$, as we can see in Figure \ref{fig:renormalization}.

\begin{figure}[ht]
\centering
\begin{tikzpicture} 

\draw (-1,0) -- (3,0);
\draw (0,-1) -- (0,5);
\draw[fill=blue!25] (0,0) -- (2,0) -- (2,2) -- (0,2) -- (0,0);
\draw (4,0) -- (8,0);
\draw (5,-1) -- (5,5); 
\draw[fill = blue!25] (5,0) -- (6,0) -- (6,4) -- (5,4) -- (5,0);
\draw[thick, ->] (3,2) -- (4,2);
\draw node at (3.5, 2.5) {$g_{-2 \log(R)}$};
\draw node at (2,-0.3) {$R$};
\draw node at (-0.3,2) {$R$};
\draw node at (6,-0.3) {$1$};
\draw node at (4.7, 4) {$R^2$};

\draw[dashed] (0,0) -- (1.8,1.8); 
\draw[dashed] (0,0) -- (1.6,0.6); 
\draw[dashed] (0,0) -- (1.4,0.2); 
\draw[dashed] (0,0) -- (0.8,0.6); 

\draw[fill] (1.8,1.8) circle [radius = 0.025]; 
\draw[fill] (1.6,0.6) circle [radius = 0.025]; 
\draw[fill] (1.4,0.2) circle [radius = 0.025]; 
\draw[fill] (0.8,0.6) circle [radius = 0.025]; 

\draw[dashed] (5,0) -- (5.9,3.6); 
\draw[dashed] (5,0) -- (5.8,1.2); 
\draw[dashed] (5,0) -- (5.7,0.4); 
\draw[dashed] (5,0) -- (5.4,1.2); 

\draw[fill] (5.9,3.6) circle [radius = 0.025]; 
\draw[fill] (5.8,1.2) circle [radius = 0.025]; 
\draw[fill] (5.7,0.4) circle [radius = 0.025]; 
\draw[fill] (5.4,1.2) circle [radius = 0.025]; 

\node at (1, -1.5) {$\Lambda(X,\omega)$};
\node at (6, -1.5) {$\Lambda(g_{-2 \log (R)}(X,\omega))$};
\end{tikzpicture}
\caption{Upon renormalizing a surface $(X,\omega)$ by applying $g_{-2 \log(R)}$, the slopes of the saddle connections of $(X,\omega)$ scale by $R^2$. }
\label{fig:renormalization}
\end{figure}
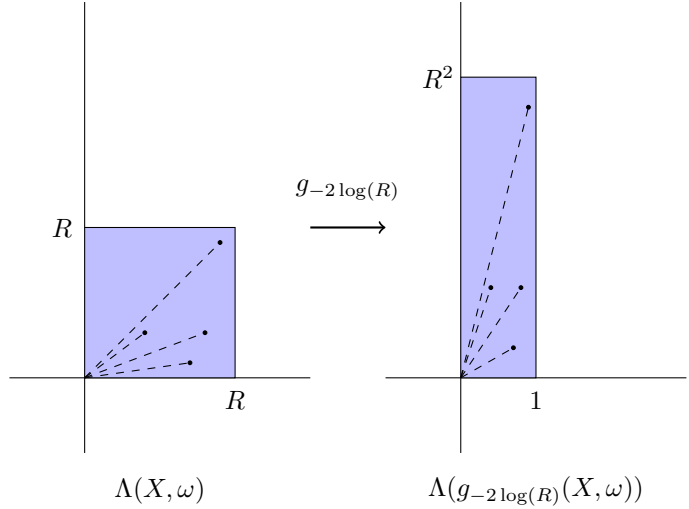

Another important observation is that the horocycle flow $$h_s = \begin{bmatrix} 1 & 0 \\ -s & 1 \end{bmatrix},$$ changes slopes of holonomy vectors by $s$. That is, $$\text{slope}(h_s (\mathbf{z})) = \text{slope}(\mathbf{z}) - s$$ for $\mathbf{z} \in \mathbb{R}^2$. As a result, slope differences are preserved by the flow $h_s$. 

Now we let the Veech group of the surface be $\text{SL}(X,\omega)$. We can then define a transversal or \textbf{Poincar\'{e} section} for the horocycle flow on $\text{SL}(2, \RR) / \text{SL}(X, \omega)$ to be the surfaces in the orbit of $(X,\omega)$ with a short horizontal saddle connection of length $\leq 1$. That is, 

$$\Omega(X,\omega) = \{g\text{SL}(X, \omega) : g \Lambda \cap \left((0,1] \times \{0\}\right) \neq \emptyset\}.$$

Then, the slope gaps of $(X,\omega)$ for holonomy vectors of $\ell_\infty$ length $\leq R$ are exactly $1/R^2$ times the set of $N(R)-1$ first return times to $\Omega(X,\omega)$ of the surface $g_{-2 \log(R)} (X,\omega)$ under the horocycle flow $h_s$ for $s \in [0,R^2]$. Here, we are thinking of return times as the amount of time between each two successive times that the horocycle flow returns to the Poincar\'{e} section. In this way, the slope gaps of $(X,\omega)$ are related to the return times of the horocycle flow to the Poincar\'{e} section. Summarizing, since $\mathbb{G}_R(X,\omega)$ is the set of slope gaps renormalized by $R^2$, we have that 
$$\mathbb{G}_R (X, \omega) = \{\text{first } N(R)-1 \text{ return times of } g_{-2\log(R)}(X,\omega) \text{ to } \Omega(X,\omega) \text{ under } h_s\}.$$

For a point $z$ in the Poincar\'{e} section $\Omega(X,\omega)$, we denote by $R_h(z)$ the return time of $z$ to $\Omega(X,\omega)$ under the horocycle flow. Then, as one lets $R \rightarrow \infty$, this renormalization procedure gives us that 
$$\lim_{R\rightarrow \infty} \frac{|\mathbb{G}_R(X,\omega) \cap (a,b)|}{N(R)} = \mu\{z \in \Omega(X,\omega) : R_h(z) \in (a,b)\},$$ where $\mu$ is the unique ergodic probability measure on $\Omega(X,\omega)$ for which the first return map under $h_s$ is not supported on a periodic orbit. Computing the slope gap distribution then reduces to finding a Poincar\'{e} section for the horocycle flow on $\text{SL}(2,\RR)/\text{SL}(X, \omega)$, a suitable measure on this Poincar\'{e} section, and the distribution function for the first return map on the Poincar\'{e} section. 

We note that this last point also makes it clear that every surface in the $\text{SL}(2, \mathbb{R})$ orbit of a Veech surface has the same slope gap distribution. We also note that scaling the surface by $c$ scales the gap distribution from $f(x)$ to $\frac{1}{c^4} f\left(\frac{x}{c^2} \right)$ (see \cite{UW} for a proof of this latter fact).

\subsection{Computing Gap Distributions for Veech surfaces} 
\label{subsec:algorithm}

In \cite{UW}, Uyanik and Work developed a general algorithm for computing the slope gap distribution for Veech surfaces. In particular, their algorithm finds a parameterization for the Poincar\'{e} section of any Veech surface and calculates the gap distribution by examining the first return time of the horocycle flow to this Poincar\'{e} section. In this section, we'll go over the basics of this algorithm. For more details about this algorithm as well as a proof of why it works, please see Uyanik and Work's original paper. 

We start by supposing that $(X, \omega)$ is a Veech surface with $n < \infty$ cusps. Then, we let $\Gamma_1, \ldots, \Gamma_n$ be representatives of the conjugacy classes of maximal parabolic subgroups of $\text{SL}(X,\omega)$. We are going to find a piece of the Poincar\'{e} section for each parabolic subgroup $\Gamma_i$. The idea here is that the set of shortest holonomy vectors of $(X,\omega)$ in each direction breaks up into $\bigcup_{i=1}^n \text{SL}(X, \omega) \mathbf{v}_i$ where the $\mathbf{v}_i$ vectors are in the eigendirections of the generators of each $\Gamma_i$.

The Poincar\'{e} section is given by those elements $g \in \text{SL}(X, \RR) / \text{SL}(X,\omega)$ such that $g(X,\omega)$ has a short (length $\leq 1$) horizontal holonomy vector: 
$$\Omega(X,\omega) = \{g\text{SL}(X, \omega) : g \Lambda \cap \left((0,1] \times \{0\}\right) \neq \emptyset\},$$ where $\Lambda$ is the set of holonomy vectors of $(X,\omega)$. Up to the action of $\text{SL}(X,\omega)$, these short horizontal holonomy vectors are then just $g \mathbf{v}_i$ for a unique $\mathbf{v}_i$. 

So $\Omega(X,\omega)$ then breaks up into a piece for each $\Gamma_i$, which we can parametrize as follows depending on whether $-I \in \text{SL}(X,\omega)$. 

\textbf{Case 1: $-I \in \text{SL}(X,\omega)$.}
In this case, $\Gamma_i \cong \ZZ \oplus \ZZ/2\ZZ$ and we can choose a generator $P_i$ for the infinite cyclic factor of $\Gamma_i$ that has eigenvalue $1$. Up to possibly replacing $P_i$ with its inverse, we have that there exists a $C_i \in \text{SL}(2, \RR)$ such that $$S_i = C_i P_i C_i^{-1} = \begin{bmatrix} 1 & \alpha_i \\0 & 1\end{bmatrix}$$ for some $\alpha_i > 0$ and that $C_i (X, \omega)$ has a shortest horizontal holonomy vector of $(1,0)$. The piece of the Poincar\'{e} section associated to $\Gamma_i$ is then parametrized by matrices $$\mathbf{M}_{a,b} = \begin{bmatrix} a & b \\ 0 & 1/a \end{bmatrix}$$  with $-1 \leq a < 0$ or $0 < a \leq 1$ so that $\mathbf{M}_{a,b} C_i (X, \omega)$ has a short horizontal holonomy vector $(|a|,0)$. Furthermore, since $S_i$ and $-I$ are in the Veech group of $C_i (X, \omega)$, we need to quotient out the full set of $\mathbf{M}_{a,b}$ matrices by the subgroup generated by $S_i$ and $-I$. The result is that the Poincar\'{e} section piece associated to $\Gamma_i$ can be parametrized by $$\Omega_i = \{ (a,b) \in \RR^2 : 0 < a \leq 1 \text{ and } 1 - (\alpha_i)a < b \leq 1\},$$ where each $(a,b) \in \Omega_i$ corresponds to $g \text{SL}(X,\omega)$ for $g = \mathbf{M}_{a,b} C_i$. 

\begin{remark} 
\label{remark:poincare_choice} 
We make a remark here that while $\Omega_i$ is defined in this specific way in Uyanik and Work's paper, there is actually a lot more freedom in defining $\Omega_i$. We just need to choose a fundamental domain for the $\mathbf{M}_{a,b}$ matrices under the action of $\langle S_i, -I\rangle$. For any $m, c \in \mathbb{R}$, another such fundamental domain is 
$$\Omega_i = \{(a,b) \in \RR^2 : 0 < a \leq 1 \text{ and } ma+c - (\alpha_i)a < b \leq ma + c\}.$$ 
That is, instead of choosing $\Omega_i$ to be a triangle whose top line is the line $b=1$ for $0 < a \leq 1$, we choose $\Omega_i$ to be a triangle whose top line is the line $b = ma + c$ for some slope $m$ and $b$-intercept $c$. We see the distinction between these two Poincar\'{e} section pieces in Figure \ref{fig:poincare_sections}.

\begin{figure}[ht]
\centering
\begin{tikzpicture}
\draw [fill = blue!25] (0,3) -- (3,3) -- (3,0) -- (0,3);
\node at (4.5,0) {$a$};
\node at (0,4.5) {$b$}; 
\node at (-.3, 3) {$1$};
\draw[thick, <->] (-1, 0) -- (4,0); 
\draw[thick, <->] (0,-2) -- (0,4); 
\draw (3,-.1) -- (3,.1); 
\node at (3.3,.3) {$1$}; 

\node at (1.7, 3.3) {$b = 1$}; 

\node at (1.4, -.5) {$b = 1 - (\alpha_i)a$}; 
\end{tikzpicture}
\hspace{1cm}
\begin{tikzpicture}
\draw [fill = blue!25] (0,1) -- (3,-1) -- (3,-4) -- (0,1);
\node at (4.5,0) {$a$};
\node at (0,2.5) {$b$}; 
\node at (-.3, 1) {$c$};
\draw[thick, <->] (-1, 0) -- (4,0); 
\draw[thick, <->] (0,-4) -- (0,2); 
\draw (3,-.1) -- (3,.1); 
\node at (3,.3) {$1$}; 

\node at (1.8, .7) {$b = m a + c$}; 

\node at (.75, -3) {$b = (m-\alpha_i)a +c$}; 
\end{tikzpicture}
\caption{Two possible Poincar\'{e} section pieces $\Omega_i$.}
\label{fig:poincare_sections}
\end{figure}
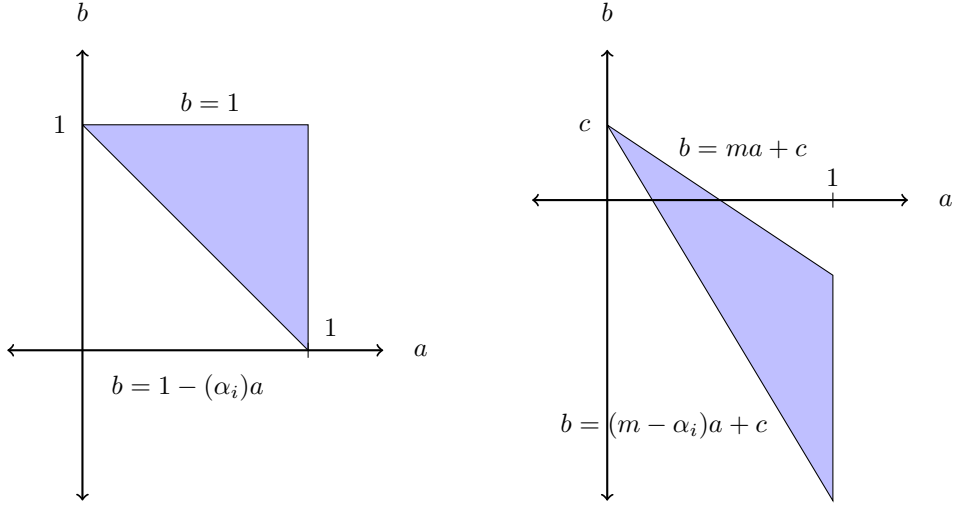

We further note that we can make similar modifications to $\Omega_i$ in Case 2 below as well. In this case, there will be another triangle with $a < 0$, and we have the freedom to choose the top line of the triangles with $a > 0$ and $a < 0$ independently.  These modifications will be integral in our finiteness proofs. 
\end{remark}

\textbf{Case 2: $-I \not \in \text{SL}(X,\omega)$.}
This case breaks up into two subcases depending on whether the generator $P_i$ of $\Gamma_i \cong \ZZ$ has eigenvalue $1$ or $-1$. 

If $P_i$ has eigenvalue $1$, then we again can find $$S_i = C_i P_i C_i^{-1} = \begin{bmatrix} 1 & \alpha_i \\0 & 1\end{bmatrix}$$ for some $\alpha_i > 0$ and that $C_i (X, \omega)$ has a shortest horizontal holonomy vector of $(1,0)$. We again have that the matrices $\mathbf{M}_{a,b}$ parameterize the Poincar\'{e} section piece, but now we only can quotient out by the subgroup generated by $S_i$. The result is that the Poincar\'{e} section piece associated to $\Gamma_i$ can be parameterized by $$\Omega_i = \{ (a,b) \in \RR^2 : 0 < a \leq 1 \text{ and } 1 - (\alpha_i)a < b \leq 1\} \bigcup \{ (a,b) \in \RR^2 : -1 \leq a < 0 \text{ and } 1 + (\alpha_i)a < b \leq 1\},$$ where each $(a,b) \in \Omega_i$ corresponds to $g \text{SL}(X,\omega)$ for $g = \mathbf{M}_{a,b} C_i$. 

When $P_i$ has eigenvalue $-1$, we can only find  $C_i \in \text{SL}(2, \RR)$ such that $$S_i = C_iP_iC_i^{-1} = \begin{bmatrix} -1 & \alpha_i \\ 0 & -1\end{bmatrix},$$ where $\alpha_i > 0$ and $C_i(X,\omega)$ has a shortest horizontal holonomy vector of $(1,0)$. We again quotient out our set of $\mathbf{M}_{a,b}$ matrices by the subgroup generated by $S_i$. The resulting Poincar\'{e} section piece associated to $\Gamma_i$ can be parameterized by $$\Omega_i = \{(a,b) \in \RR^2 : 0 < a \leq 1 \text{ and } 1 - (2 \alpha_i)a < b \leq 1\},$$ where each $(a,b) \in \Omega_i$ corresponds to $g\text{SL}(X,\omega)$ for $g = \mathbf{M}_{a,b}C_i$. 

Having established what each piece of the Poincar\'{e} section associated to each $\Gamma_i$ looks like, we also need to find the measure on the whole Poincar\'{e} section. The measure on the Poincar\'{e} section is the unique ergodic measure $\mu$ on $\Omega(X,\omega)$, which is a scaled copy of the Lebesgue measure on each of these pieces $\Omega_i$ of $\RR^2$. 

Upon finding the Poincar\'{e} section pieces, the return time function function of the horocycle flow at a point $\mathbf{M}_{a,b} C_i (X,\omega)$ is the smallest positive slope of a holonomy vector of $\mathbf{M}_{a,b}C_i(X,\omega)$ which short horizontal component. That is, if $\mathbf{v} = (x,y)$ is the holonomy vector of $C_i(X,\omega)$ such that $\mathbf{M}_{a,b}(x,y)$ is the holonomy vector on $\mathbf{M}_{a,b}(x,y)$ with the smallest positive slope among all holonomy vectors with a horizontal component of length $\leq 1$, 
then the return time function at that point $(a,b) \in \Omega_i$ in the Poincar\'{e} section is given by the slope of $\mathbf{M}_{a,b}(x,y)$, which is $$\frac{y}{a(ax+by)}.$$ 

We call such a vector $\mathbf{v} = (x,y)$ a \textbf{winner} or \textbf{winning saddle connection}. We note that while technically $\mathbf{v}$ is the holonomy vector of a saddle connection, we will often use the terms holonomy vector and saddle connection interchangeably. Our proof that the slope gap distribution of a Veech surface has finitely many points of non-analyticity will rely on us showing that each piece $\Omega_i$ of the Poincar\'{e} section has finitely many winners. 

Each such $\mathbf{v}$ would then be a winner on a convex polygonal piece of $\Omega_i$. Furthermore, the cumulative distribution function of the slope gap distribution would then be given by areas between the hyperbolic return time function level curves and the sides of these polygons, and would therefore be piecewise real analytic with finitely many points of non-analyticity. 

\subsection{Examples and Difficulties}
\label{subsec:examples}

In this section, we will give an example of difficulties that arise from the choice of parameterization of the Poincar\'{e} section. In particular, it is possible for there to be infinitely many winning saddle connections under certain parameterizations, but only finitely many different winners under a different parameterization. For full computations of a gap distribution we refer to \cite{ACL} and \cite{UW}.

%
We will take the surface $\scrL$ in Figure \ref{fig:L-example} and analyze the winning saddle connection on the component $\Omega_1$ of the Poincar\'{e} section corresponding to the parabolic subgroup of $\text{SL}(\scrL)$ generated by $\begin{bmatrix} 1 & 1 \\ 0 & 1 \end{bmatrix}$. $\scrL$ is a $7$-square square-tiled surface with a single cone point.

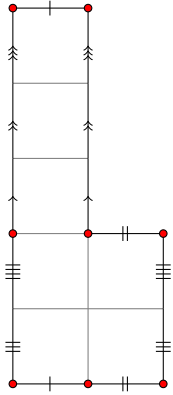
\begin{figure}[ht]
\centering
\begin{tikzpicture}

\draw[gray] (1,0) -- (1,1) -- (1,2);
\draw[gray] (0,1) -- (1,1) -- (2,1);
\draw[gray] (0,2) -- (1,2);
\draw[gray] (0,3) -- (1,3);
\draw[gray] (0,4) -- (1,4);

\draw[-|-] (0,0) -- (1,0); 
\draw[-||-] (1,0) -- (2,0); 
\draw[-|-] (0,5) -- (1,5); 
\draw[-||-] (1,2) -- (2,2); 

\draw[-|||-] (0,0) -- (0,1);
\draw[-|||-] (2,0) -- (2,1);

\draw[-||||-] (0,1) -- (0,2);
\draw[-||||-] (2,1) -- (2,2);

\draw[ma={>}] (0,2) -- (0,3);
\draw[ma={>}] (1,2) -- (1,3);

\draw[ma={>>}] (0,3) -- (0,4);
\draw[ma={>>}] (1,3) -- (1,4);

\draw[ma={>>>}] (0,4) -- (0,5);
\draw[ma={>>>}] (1,4) -- (1,5);

\draw[fill = red] (0,0) circle [radius = 0.05];
\draw[fill = red] (1,0) circle [radius = 0.05];
\draw[fill = red] (2,0) circle [radius = 0.05];
\draw[fill = red] (2,2) circle [radius = 0.05];
\draw[fill = red] (1,2) circle [radius = 0.05];
\draw[fill = red] (0,2) circle [radius = 0.05];
\draw[fill = red] (0,5) circle [radius = 0.05];
\draw[fill = red] (1,5) circle [radius = 0.05];

\end{tikzpicture} 
\caption{The surface $\scrL$ with cone point in red}
\label{fig:L-example}
\end{figure}

Since $\begin{bmatrix} 1 & 1 \\ 0 & 1 \end{bmatrix}$ is in the Veech group, and $\scrL$ has a length $1$ horizontal saddle connection, the corresponding piece of the Poincar\'{e} section $\Omega_1$ can be parametrized by matrices $\mathbf{M}_{a,b} = \begin{bmatrix} a & b \\ 0 & a^{-1} \end{bmatrix}$ with $0 < a \le 1$ and $1 - a < b \le 1$. Notice that $\scrL$ has all saddle connections with coordinates $(n,2)$ and $(n,3)$ for $n \in \ZZ$, and no saddle connection with $y$-coordinate $1$.

\begin{prop}
\label{prop:y_coord_winner}
In a neighborhood of the point $(0,1)$ on $\Omega_1$, the winning saddle connection always has $y$-coordinate $2$.
\end{prop}
\begin{proof}
Take a saddle connection $\mathbf{v} = (n,k)$ with $k > 0$ so that $\mathbf{M}_{a,b}\mathbf{v}$ with horizontal component $\leq 1$. We will show that if $k > 2$ and $a < \frac{1}{3}$ there is a saddle connection $\mathbf{w} = (m,2)$ so that the slope of $\mathbf{M}_{a,b}\mathbf{w}$ is less than the slope of $\mathbf{M}_{a,b}\mathbf{v}$, and so that $\mathbf{M}_{a,b}\mathbf{w}$ has short horizontal component. Since there are no saddle connections with $k = 1$, this implies that the winning saddle connection must have $y$-coordinate $2$.

The $x$-coordinate of $\mathbf{M}_{a,b}\mathbf{w}$ is $ma+2b$. Since $\scrL$ has all $(m,2)$ saddle connections, we may choose an $m$ so that $1-a < ma+2b \le 1$. The condition that $\textup{slope}(\mathbf{M}_{a,b}\mathbf{w}) < \textup{slope}(\mathbf{M}_{a,b}\mathbf{v})$, rearranges to

\[ (na+kb) < \frac{k}{2}(ma+2b) \]

If $k > 2$, then since $\scrL$ is square-tiled we have that $k \ge 3$, and when $a < 1/3$ we have that $ma+2b > 1-a \ge \frac{2}{3}$, thus $\frac{k}{2}(ma+2b) > 1$. Since $\mathbf{M}_{a,b}\mathbf{v}$ has a short horizontal component, we know that $na+kb \le 1$, so the above inequality is always true.
\end{proof}

Let $A_m$ be the region where the saddle connection $(-m,2)$ is the winning saddle connection. By Proposition \ref{prop:y_coord_winner}, in the top left corner of $\Omega_1$, $A_m$ is the region where $\mathbf{M}_{a,b}(-m,2) = (2b-ma,2a^{-1})$ has smallest slope among all saddle connections with $y$-coordinate $2$ and short horizontal component. The slope is $\frac{2a^{-1}}{2b-ma}$, so minimizing the slope is equivalent to maximizing $2b-ma$ with the constraint that $2b-ma \le 1$, or in other words, $-m = \lfloor \frac{1-2b}{a} \rfloor$. But as $a \rightarrow 0$, $b \rightarrow $1, so $-m \sim \frac{-1}{a} \rightarrow - \infty$. This implies that there infinitely many saddle connections that occur as winners in the top left corner of the Poincar\'{e} section.

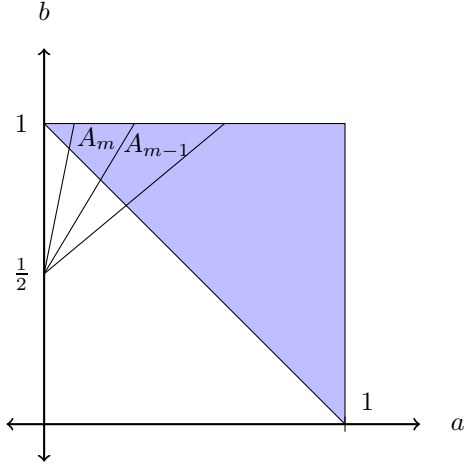
\begin{figure}[ht]
\centering
\begin{tikzpicture}

\draw [fill = blue!25] (0,4) -- (4,4) -- (4,0) -- (0,4);
\node at (5.5,0) {$a$};
\node at (0,5.5) {$b$}; 
\node at (-.3, 4) {$1$};
\draw[thick, <->] (-.5, 0) -- (5,0); 
\draw[thick, <->] (0,-.5) -- (0,5); 
\draw (4,-.1) -- (4,.1); 
\node at (4.3,.3) {$1$}; 

\draw (0,2) -- (.4,4);
\draw (0,2) -- (1.2,4);
\draw (0,2) -- (2.4,4);

\node at (.7 ,3.8) {$A_m$};
\node at (1.5,3.7) {$A_{m-1}$};
\node at (-.3, 2) {$\frac12$};

\end{tikzpicture} 
\caption{Regions $A_m$ where $(-m,2)$ is a winner}
\label{fig:infinitely_many_pieces}
\end{figure}

By Remark \ref{remark:poincare_choice} in Section \ref{subsec:algorithm}, we notice that we can change the parameterization of the Poincar\'{e} section. One problem in our previous parameterization was that there were infinitely many winners in the upper left hand corner $(0,1)$ of our Poincar\'{e} section. To fix this, we will change our parameterization so that the  upper left corner is at $(0,1/2)$ and the slope of the top line of our Poincar\'{e} section triangle is nicely compatible with the $(1,2)$ holonomy vector. This will ensure that there are finitely many winners in the top left corner and will result in finitely many winners across the entire Poincar\'{e} section. We will prove that we can always do this for arbitrary Veech surfaces in  Section \ref{sec:main}.

\begin{figure}[ht]
\centering
\begin{tikzpicture}

\fill [blue!25] (0,1.5) -- (3,0) -- (3,-1) -- (1.5,0) -- (0,1.5);
\fill [yellow!25] (3,-1) -- (3,-1.5) -- (1.5,0);
\fill [red!25] (0,1.5) -- (3,-1.5) -- (3,-3) -- (0,1.5);

\draw (3,0) -- (0,1.5) -- (1.5,0) -- (3,0) -- (3,-1) -- (1.5,0) -- (3,-1.5) -- (3,-1) -- (3,-3) -- (0,1.5);

\node at (4.5,0) {$a$};
\node at (0,2.5) {$b$}; 
\node at (-.3, 1.5) {$\frac{1}{2}$};
\node at (-.3, -3) {$-1$};
\draw[thick, <->] (-1, 0) -- (4,0); 
\draw[thick, <->] (0,-3.5) -- (0,2); 
\node at (3.3,.3) {$1$}; 

\node[blue] at (3.9, -.5) {$(1,2)$ wins}; 
\node[RedOrange] at (3.9, -1.25) {$(2,3)$ wins};
\node[red] at (3.9, -2.25) {$(2,2)$ wins};

\end{tikzpicture} 
\caption{The new Poincar\'{e} section breaks up into $3$ pieces, with saddle connections (1,2) winning in the blue region, (2,3) in the yellow region, and (2,2) in the red region}
\label{fig:better_parameterization}
\end{figure}
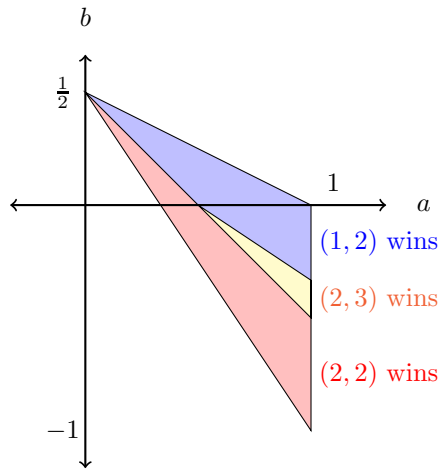

We will use the parameterization $0 < a \le 1$ and $\frac{1}{2}-\frac{3}{2}a < b \le \frac{1}{2} - \frac{1}{2}a$. This parameterization is chosen to ensure that the saddle connection $(1,2)$ of $\scrL$ wins in a neighborhood of the top line segment, which prevents the problem that arises in the previous parameterization.

In this case the only winners are the $(1,2)$, $(2,3)$ and $(2,2)$ saddle connections on $\scrL$. 

\begin{enumerate}
    \item $(1,2)$ wins in the region $$\left\{ (a,b) \bigg| 0 < a \le \frac{1}{2}, \frac{1}{2}-a < b \le \frac{1}{2} - \frac{1}{2}a \textup{ and } \frac{1}{3}-\frac{2}{3}a < b \right\}.$$
    \item $(2,3)$ wins in the region 
    
    $$ \left\{ (a,b) \bigg| \frac{1}{2} < a \le 1, \frac{1}{2}-a < b \le \frac{1}{3}-\frac{2}{3}b \right\}. $$
    \item $(2,2)$ wins in the region
    
    $$ \left\{ (a,b) \bigg| 0 < a \le 1, \frac{1}{2} -\frac{3}{2}a < b \le \frac{1}{2}-a \right\}.$$
\end{enumerate}

To see this, notice that the saddle connection $(x,y)$ is the winner at $(a,b)$ if $\mathbf{M}_{a,b}(x,y)$ has smallest positive slope amongst all saddle connections with short horizontal component. $\mathbf{M}_{a,b}(x,y)$ has short horizontal in the region $0 < a \le 1$, $\frac{-x}{y}a < b \le \frac{1}{y} - \frac{x}{y}a$. Minimizing the slope at $(a,b)$ is equivalent to maximizing $\frac{x}{y}$ over all saddle connections with a short horizontal component.

Working out the exact winners then comes down to casework. In this case, $\mathbf{M}_{(a,b)}(m,2)$ never has a short horizontal component for $m > 2$ and $(a,b)$ in the Poincar\'{e} section, and simple casework shows where $(1,2)$ and $(2,2)$ are the winners. For saddle connections with $y$-coordinate greater than $2$, we need to understand those with $x/y > \frac{1}{2}$ which can potentially win against $(1,2)$ or $(2,2)$. $(2,3)$ wins in the yellow region as $(2,2)$ does not have a short horizontal component for $(a,b)$ in that region. All other saddle connection with $y = 3$ and $x \ge 3$ do not have short horizontal component in the Poincar\'{e} section. For $y \ge 4$, a similar analysis shows that none of the saddle connections can appear as winners, giving the result.
%

\section{Main Theorem}
\label{sec:main}

In Section \ref{subsec:examples}, we examined the 7-square tiled surface $\scrL$ and saw that in one parameterization, it looked like the Poincar\'{e} section would admit infinitely many winning saddle connections and therefore give the possibility of infinitely many points of non-analyticity in the slope gap distribution. However, when we strategically chose a different parameterization of this piece of the Poincar\'{e} section, there were now only finitely many winners. Thus, this piece of the Poincar\'{e} section could only contribute finitely many points of non-analyticity to the slope gap distribution. We could then have repeated this process for the other pieces of the Poincar\'{e} section. 

This is one of the key ideas of the main theorem of this paper: 

\finite

This section is devoted to the proof of this theorem. We will begin by giving an outline of the proof and then will dive into the details of each step. 

\subsection{Outline}
Let us begin with an outline of the proof of the main theorem, Theorem \ref{thm:finite}. The idea is that after choosing strategic parameterizations of each piece of the Poincar\'{e} section of a Veech translation surface $(X,\omega)$, we will use compactness arguments to show that there are finitely many winners on each piece. 

\begin{enumerate}
    \item We begin with a Veech translation surface $(X,\omega)$ and focus on a piece of its Poincar\'{e} section corresponding to one maximal parabolic subgroup in $\text{SL}(X,\omega)$. Up to multiplication by an element of $\text{GL}(2,\mathbb{R})$, we will assume that the generator of the parabolic subgroup has a horizontal eigenvector and $(X,\omega)$ has a horizontal saddle connection of length $1$. Based on properties of the saddle connection set of $(X,\omega)$, we strategically choose a parameterization $T_X$ of this Poincar\'{e} section piece. $T_X$ will be some triangle in the plane. 
    \item For any saddle connection $\mathbf{v}$ of $(X,\omega)$, we will define a strip $S_\Omega(\mathbf{v})$ that gives a set of points$(a,b) \in \RR_{> 0} \times \RR$ where $\mathbf{v}$ is a potential winning saddle connection on the surface $\mathbf{M}_{a,b}(X,\omega) \in T_X$. We will start by showing various properties of these strips that we will make use of later on in the proof. 
    \item We will then show that there for every point $(a,b) \in T_X$ with $b > 0$, $(a,b)$ has an open neighborhood with finitely many winning saddle connections.
    \item We then move on to show that for points $(a,b) \in T_X$ with $b \leq 0$ that are not on the right boundary $a=1$ of $T_X$, $(a,b)$ has an open neighborhood with finitely many winning saddle connections. 
    \item Next, we show that on the right boundary $a=1$ of $T_X$, there are finitely many winning saddle connections. 
    \item Using the finiteness on the right boundary, we show that for any point $(a,b) \in T_X$ with $a=1$, $(a,b)$ has an open neighborhood with finitely many winning saddle connections. 
    \item By compactness of $T_X$, there is a finite cover of $T_X$ with the open neighborhoods of points $(a,b) \in T_X$ that we found in our previous steps. Since each of the these open neighborhoods had finitely many winners, we find that there are finitely many winning saddle connections across all of $T_X$. 
    \item Finally, we show that finitely many winners on each piece of the Poincar\'{e} section implies finitely many points of non-analyticity of the slope gap distribution. 
\end{enumerate}

\subsection{Proof}

Using the method of \cite{UW} outlined in Section \ref{subsec:algorithm}, it will suffice to show that every piece of the Poincar\'{e} section can be chosen so that there are only finitely many winning saddle connections. For most of the arguments in this section, we will fix a piece of the Poincar\'{e} section and will work exclusively with it.

We recall that there is a piece of the Poincar\'{e} section for each conjugacy class of maximal parabolic subgroup in $\text{SL}(X,\omega)$. We will now fix such a maximal parabolic subgroup $\Gamma_i$ and work with the corresponding component of the Poincar\'{e} section. Without loss of generality we may assume that $(X,\omega)$ has a horizontal saddle connection with $x$-component $1$ and that $\Gamma_i$ is generated by 

\[ P_i = \begin{bmatrix} 1 & \alpha_i \\ 0 &1 \end{bmatrix} \]

Using the notation of Section \ref{subsec:algorithm}, this is essentially replacing $(X,\omega)$ with $C_i(X,\omega)$.

Since $(X,\omega)$ is a Veech surface with a horizontal saddle connection, it has a horizontal cylinder decomposition (\cite{HS}), and therefore, for all $a \in \RR$ there are only finitely many heights $0 \le h \le a$ so that $(X,\omega)$ has a saddle connection with $y$-component $h$. Let $y_0 > 0$ be the shortest vertical component of a saddle connection on $(X,\omega)$, and let $x_0 > 0$  be the shortest horizontal component of a saddle connection at height $y_0$. Our first step is to use this saddle connection to give a parameterization of the Poincare section that is adaptaed to the geometry of $(X,\omega)$.

By Remark \ref{remark:poincare_choice}, we can choose the following parameterization of this piece of the Poincar\'{e} section, as pictured in Figure \ref{fig:T_X}: 

\[ T_X = \left\{  (a,b) \, \bigg| \, 0 < a \le 1, \frac{1-x_0 a}{y_0}-na \le b \le \frac{1-x_0 a}{y_0} \right\}. \]

Here, $n$ is either $\alpha_i$ or $2 \alpha_i$ depending on which one is needed to fully parameterize this piece of the Poincar\'{e} section, as described in Section \ref{subsec:algorithm}. In the case where $-I \not \in \text{SL}(X,\omega)$ and $P_1$ had eigenvalue $1$, the Poincar\'{e} section has an additional triangle with $a < 0$. In particular, we can choose this triangle so that it consists of points $(-a,-b)$ for $(a,b) \in T_X$. But we note that if $\mathbf{v}$ were the winning saddle connection for $\mathbf{M}_{a,b}(X,\omega)$, then $-\mathbf{v}$ is the winning saddle connection for $\mathbf{M}_{-a,-b}(X,\omega)$, and hence if proving that there are only finitely many winners, it suffices to consider only the portion with $a > 0$.

\begin{figure}[ht]
\centering
\begin{tikzpicture}
\draw [fill = blue!25] (0,1) -- (3,-1) -- (3,-4) -- (0,1);
\node at (4.5,0) {$a$};
\node at (0,2.5) {$b$}; 
\node at (-.3, 1) {$\frac{1}{y_0}$};
\draw[thick, <->] (-1, 0) -- (4,0); 
\draw[thick, <->] (0,-4) -- (0,2); 
\draw (3,-.1) -- (3,.1); 
\node at (3,.3) {$1$}; 

\node at (1.7, .8) {$b = \frac{-x_0}{y_0} a + \frac{1}{y_0}$}; 

\node at (.7, -3) {$b = (\frac{-x_0}{y_0}-n) a + \frac{1}{y_0}$}; 
\end{tikzpicture}
\caption{A Poincar\'{e} section piece for $(X,\omega)$ with $y_0 > 0$ the shortest vertical component and $x_0>0$ the corresponding shortest component of a saddle connection on $(X, \omega)$. }
\label{fig:T_X}
\end{figure}
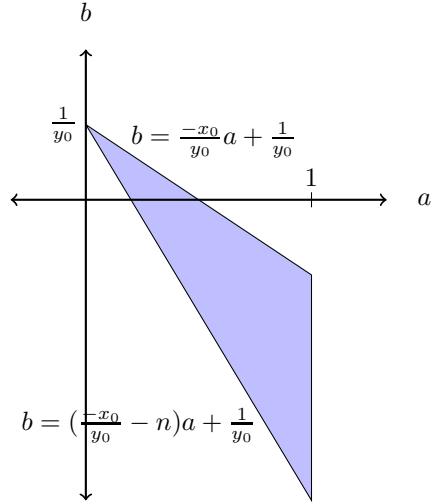

Our goal now is to prove that the return time function is piecewise real analytic with finitely many pieces. We will do so by proving that there are finitely many winning saddle connections $\mathbf{v}_1, \ldots, \mathbf{v}_n \in \Lambda(X,\omega)$ such that each point $(a,b) \in T_X$ has a winner $\mathbf{M}_{a,b}\mathbf{v}_i$ for some $1 \leq i \neq n$.  We will repeat this for every $T_X$ corresponding to each maximal parabolic subgroup.

To achieve this goal, we will first define an auxiliary set that will help us understand for what points $(a,b) \in T_X$ a particular $\mathbf{v} \in \Lambda(X,\omega)$ is a candidate winner. By a candidate winner, we mean that $\mathbf{M}_{a,b} \mathbf{v}$ has a positive $x$ coordinate at most $1$ and a positive $y$ coordinate. If $\mathbf{v} = (x,y)$, the $x$-coordinate condition is the condition that $0 < ax + by \leq 1$. We also note that for $\mathbf{M}_{a,b}(x,y)$ to be be a winner, we need that $a^{-1}y > 0$. Since $a> 0$ on $T_X$, the latter condition reduces to saying that $y > 0$.

\begin{defn}
\label{def:strip_omega}
Given a saddle connection $\mathbf{v} = (x,y)$ with $y > 0$, we define $\mathcal{S}_\Omega(\mathbf{v})$ as the strip of points $(a,b) \in \mathbb{R}_{> 0} \times \mathbb{R}$ such that $0 < ax+by\leq 1$. This corresponds to the set of surfaces $\mathbf{M}_{a,b}(X,\omega)$ for which $\mathbf{M}_{a,b}\mathbf{v}$ is a potential winning saddle connection. 
\end{defn}

\begin{figure}[ht]
\centering
\begin{tikzpicture}
\draw [fill = blue!25, blue!25] (0,0) -- (4,-3) -- (4,-2) -- (0,1) -- (0,0);
\node at (4.5,0) {$a$};
\node at (0,2.5) {$b$}; 
\node at (-.3, 1) {$\frac{1}{y}$};
\node at (1.5, .3) {$\frac{1}{x}$};
\draw[thick, <->] (-1, 0) -- (4,0); 
\draw[thick, <->] (0,-3) -- (0,2); 
\draw[dashed] (0,0) -- (4,-3); 
\draw (0,1) -- (4,-2);
\end{tikzpicture}
\caption{A strip $\mathcal{S}_\Omega(\mathbf{v})$ for $\mathbf{v} = (x,y)$. Here $y > 0$. The slope of the upper and lower lines of the strip is $-\frac{x}{y}$.}
\end{figure}
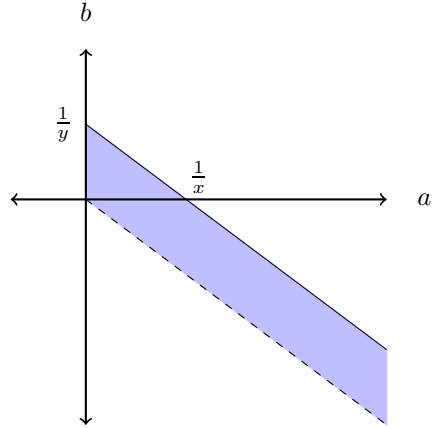

Let us note some properties of these strips $\mathcal{S}_\Omega(\mathbf{v})$ that we will be using repeatedly in our proofs. We recall that we are assuming without loss of generality that $(X,\omega)$ has a short horizontal saddle connection of length $1$. Considering the particular piece $T_X$ of the Poincar\'{e} section, we recall that $T_X$ is parametrized by matrices $\mathbf{M}_{a,b} = \begin{bmatrix} a & b \\ 0 & a^{-1} \end{bmatrix}$ so that $\mathbf{M}_{a,b}(X,\omega)$ has a short horizontal saddle connection of length $\leq 1$. 

Then, because $(X,\omega)$ is a Veech surface, it breaks up into horizontal cylinders and therefore there exists an $y_0 > 0$ such that there is a saddle connection with height $y_0$ and furthermore that every saddle connection with positive height has height $\geq y_0$. 

With these assumptions in in place, we note the following useful properties of the strips $\mathcal{S}_\Omega(\mathbf{v})$ that are used implicitly throughout the proof:

\begin{enumerate}
    \item The strip $\mathcal{S}_\Omega(\mathbf{v})$ for $\mathbf{v} = (x,y)$ is sandwiched between a solid line that intersects the $b$-axis at $\frac{1}{y}$ and a dotted line that intersects the $b$-axis at $0$. Both lines have slope $-\frac{x}{y}$. We also know that $y \geq y_0$, so $\frac{1}{y} \leq \frac{1}{y_0}$.  

    \item Fixing any $c > 0$, there are only finitely many $y$ coordinates of saddle connections $\mathbf{v}$ such that $S_\Omega(\mathbf{v})$ intersects the $y$-axis at any point $\geq c$. 
    
    This is because $(X, \omega)$ being a Veech surface and having a horizontal saddle connection implies that the surface breaks up into finitely many horizontal cylinders of heights $h_1, \ldots, h_n$ and every saddle connection with positive $y$-component must have a $y$-component that is a nonnegative linear combination of these $h_i$s. Since there are finitely many such $y$ values $\leq 1/c$, there are finitely many strips that intersect the $y$ axis at points $\geq c$.

    \item At a particular point $(a,b) \in T_X$, the winner is the saddle connection $\mathbf{v} = (x,y) \in \Lambda(X,\omega)$ such that $\mathbf{M}_{a,b}\mathbf{v} = (ax+by, a^{-1}y)$ has the least slope among those saddle connections satisfying $0 <ax+by\leq 1$ and $a^{-1}y > 0$. Since $a>0$ for any point in $\Omega_i^M$, this corresponds to the saddle connection with the greatest $\frac{x}{y}$ with $y > 0$. 
    
    In terms of our strips, we're fixing the point $(a,b)$ and looking for the strip $\mathcal{S}_\Omega(\mathbf{v})$ that contains $(a,b)$ and has the least slope, since each strip has slope $-\frac{x}{y}$.
    
    \item For any given $y > 0$, there are only finitely many saddle connection vectors $\mathbf{v} = (x,y)$ of $(X,\omega)$ with $x \geq 0$ such that $S_\Omega(\mathbf{v})$ intersects $T_X$. 
    
    This is because $S_\Omega(\mathbf{v})$ does not intersect $T_X$ for $\frac{x}{y}$ larger than some constant $C$ that depends on $T_X$ and $y$. Specifically, we can let $C = \frac{x_0}{y_0} + n$, the negative of the slope of the bottom line that defines the triangle $T_X$. Since the saddle connection set is discrete, there are finitely many $x \geq 0$ for a given $y$ such that $\frac{x}{y} \leq C$.  
\end{enumerate}

With these facts established, let us first show a lemma that reestablishes the known fact that the return time of the horocycle flow to any point $(a,b) \in T_X$ is finite and that will be useful in proving Lemma \ref{lem:finiteboundary} later. 
 
\begin{lemma} 
\label{lem:winners_exist}
Let $(a,b) \in T_X$ so that $(a,0)$ is a short horizontal saddle connection of $\mathbf{M}_{a,b}(X,\omega)$. Then $\mathbf{M}_{a,b}(X,\omega)$ has saddle connections $\mathbf{v}_1 = (x_1, y_1)$ and $\mathbf{v}_2 = (x_2,y_2)$ so that $y_1, y_2 > 0$ and $0 < x_1 \le a$ and $0 \le x_2 < a$. (These saddle connections might be the same). This implies that every point in $T_X$ has a winning saddle connection, or equivalently, that every point in $T_X$ is in some strip $S_\Omega(\mathbf{v})$ for some $\mathbf{v} = (x,y)$ with $y > 0$. 
\end{lemma}

\begin{proof}
Let us take a horizontal saddle connection on our surface $\mathbf{M}_{a,b}(X, \omega)$ with holonomy vector $(a,0)$, connecting two (possibly identical) cone points $p$ and $q$. Then, we will consider developing a width $a$ vertical strip on our surface extending upward with the open horizontal segment from $p$ to $q$ as its base. Since our surface is of finite area, this vertical strip must eventually hit a cone point $r$ or come back to overlap our original open segment from $p$ to $q$. Now we're going to define our vectors $\mathbf{v}_1$ and $\mathbf{v}_2$ in each case. 

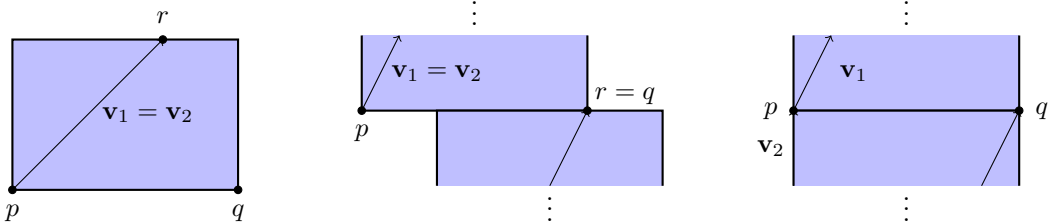
\begin{figure}[ht]
\centering
\begin{tikzpicture}
\draw[thick, fill = blue!25] (0,0) -- (3,0) -- (3,2) -- (0,2) -- (0,0); 
\draw[fill] (0,0) circle [radius = 0.05]; 
\draw[fill] (3,0) circle [radius = 0.05]; 
\draw[fill] (2,2) circle [radius = 0.05]; 
\node at (0,-.3) {$p$}; 
\node at (3,-.3) {$q$}; 
\node at (2, 2.3) {$r$}; 
\draw[->] (0,0) -- (2,2); 
\node at (1.8,1) {$\mathbf{v}_1 = \mathbf{v}_2$};
\end{tikzpicture} \hspace{1cm}
\begin{tikzpicture}
\draw[thick, fill = blue!25] (0,1) -- (0,0) -- (3,0) -- (3,1); 

\draw[thick, fill = blue!25] (1,-1) -- (1,0) -- (4,0) -- (4,-1); 

\draw[fill] (0,0) circle [radius = 0.05]; 
\draw[fill] (3,0) circle [radius = 0.05]; 
\draw[->] (0,0) -- (.5,1); 
\draw[->] (2.5,-1) -- (3,0); 
\node at (0,-.3) {$p$}; 
\node at (3.5,.2) {$r = q$}; 
\node at (1.5,1.4) {$\vdots$};
\node at (2.5,-1.2) {$\vdots$};
\node at (1,.5) {$\mathbf{v}_1 = \mathbf{v}_2$};
\end{tikzpicture}\hspace{1cm}
\begin{tikzpicture}
\draw[thick, fill = blue!25] (0,1) -- (0,0) -- (3,0) -- (3,1); 
\draw[thick, fill = blue!25] (0,-1) -- (0,0) -- (3,0) -- (3,-1); 
\draw[fill] (0,0) circle [radius = 0.05]; 
\draw[fill] (3,0) circle [radius = 0.05]; 
\draw[->] (0,0) -- (.5,1); 
\draw[->] (2.5,-1) -- (3,0); 
\draw[->] (0,-1) -- (0,0); 
\node at (-.3,0) {$p$}; 
\node at (3.3,0) {$q$}; 
\node at (1.5,1.4) {$\vdots$};
\node at (1.5,-1.2) {$\vdots$};
\node at (.8,.5) {$\mathbf{v}_1$};
\node at (-.3,-.5) {$\mathbf{v}_2$};
\end{tikzpicture}
\caption{The vectors $\mathbf{v}_1$ and $\mathbf{v}_2$ in the three different cases of vertical strip.}
\end{figure}

In the former case when the top edge of our vertical strip hits a cone point $r$ in the interior of the edge, the straight segment from $p$ to $r$ cutting through our vertical strip gives us both our saddle connections $\mathbf{v}_1=\mathbf{v}_2$. 

The latter case when the top edge of our vertical strip comes back to overlap our original open segment breaks up into two cases. If we have a complete overlap, then the saddle connection from $p$ on the bottom edge to $q$ on the top edge gives us our vector $\mathbf{v}_1$ and the saddle connection from $p$ on the bottom edge to $p$ on the top edge gives us our vector $\mathbf{v}_2$. If we have an incomplete overlap, then the top edge contains the cone point $r=p$ or $r=q$, and the saddle connection from $p$ on the bottom edge to $r$ on the top edge gives us both of our saddle connections $\mathbf{v}_1 = \mathbf{v}_2$.

In any of these cases letting $\mathbf{v} = \mathbf{M}_{a,b}^{-1} (\mathbf{v}_1)$ gives us that $S_\Omega(\mathbf{v})$ contains our initial point $(a,b)$ and $\mathbf{v}$ is a possible winning saddle connection. 
\end{proof}

The following lemma will help us show that there are finitely many winning saddle connections on certain sets in $T_X$. 

\begin{lemma} 
\label{lem:finitewinners}
Let $S$ be a closed set that is a subset of $S_\Omega(\mathbf{v})$ for $\mathbf{v} = (x,y)$ with  $y > 0$. Then, there are finitely many winning saddle connections on $S$.  
\end{lemma}

\begin{proof}
Let $S$ be a closed set contained in $S_\Omega(\mathbf{v})$ for a saddle connection $\mathbf{v} = (x,y)$ of $(X,\omega)$ with $y > 0$. By definition, we have that $\mathbf{v}$ is a potential winning saddle connection on all of $S$. That is, for any point $(a,b) \in S$, $\mathbf{M}_{a,b} \mathbf{v}$ has positive $y$ component and positive and short ($\leq 1$) $x$ component. 

 We recall that for a point $(a',b') \subset S$ to have winner $\mathbf{v'} = (x',y') \neq \mathbf{v} = (x,y)$, we need that $\mathbf{v'}$ is a saddle connection of $(X, \omega)$, $\frac{x'}{y'} > \frac{x}{y}$, and that $(a',b') \subset S_\Omega(\mathbf{v'})$. 

This corresponds to the strip $S_\Omega(\mathbf{v'})$ having a smaller slope than $S_\Omega(\mathbf{v})$ and still intersecting $S$. Given that $S$ is closed and the bottom boundary of $S_\Omega(\mathbf{v})$ is open, there exists an $h > 0$ so that the line $S$ is completely on or above the line $b = -\frac{x}{y}a + \frac{1}{h}$. Furthermore, since the left boundary of $S_\Omega(\mathbf{v})$ is open, $S$ is a positive distance away from the $y$ axis.

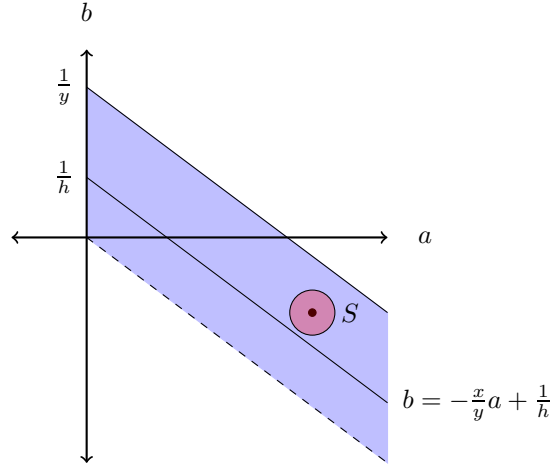
\begin{figure}[ht]
\centering
\begin{tikzpicture}
\draw [fill = blue!25, blue!25] (0,0) -- (4,-3) -- (4,-1) -- (0,2) -- (0,0);
\node at (4.5,0) {$a$};
\node at (0,3) {$b$}; 
\node at (-.3, 2) {$\frac{1}{y}$};
\node at (-.3, .8) {$\frac{1}{h}$};
\draw[thick, <->] (-1, 0) -- (4,0); 
\draw[thick, <->] (0,-3) -- (0,2.5); 
\draw[dashed] (0,0) -- (4,-3); 
\draw (0,2) -- (4,-1);

\draw[fill] (3, -1) circle [radius = 0.05];
\draw[fill = red, fill opacity = 0.3] (3,-1) circle [radius = .3];

\node at (3.5,-1) {$S$};
\draw (0,.8) -- (4,-2.2);
\node at (5.2, -2.2) {$b = -\frac{x}{y}a + \frac{1}{h}$}; 
\end{tikzpicture}
\caption{A choice of $1/h$ for a particular $S$.}
\end{figure}

Then, for $S_\Omega(\mathbf{v'})$ to intersect $S$ and for $\frac{x'}{y'} > \frac{x}{y}$, we need that $y' < h$ since otherwise the strip $S_\Omega(\mathbf{v'})$ would have $y$-intercepts $1/y' \leq 1/h$ and $0$ and would have smaller slope than that of $S_\Omega(\mathbf{v})$ and would therefore not intersect $S$. 

But since $(X,\omega)$ is a Veech surface with a horizontal saddle connection, it decomposes into finitely many horizontal cylinders. Therefore, the set of possible vertical components $y'$ of saddle connections are a discrete subset of $\mathbb{R}$ and thus, there are finitely many vertical components of saddle connections that satisfy $y' < h$. Since there are finitely many saddle connections in the vertical strip $(0,1]\times (0,\infty)$ with vertical component less than $h$, then there are finitely many possible winning saddle connections on $S$.
\end{proof}

We recall that our goal is to show that every point $(a,b) \in T_X$ has a neighborhood on which there are finitely many winners. This will allow us to use a compactness argument to prove that there are finitely many winners on all of $T_X$. Building off of the previous lemma, we show in the next lemma that certain points $(a,b) \in T_X$ have an open neighborhood on which there are finitely many winners. 

\begin{lemma}
\label{lem:interiorstrip}
Let $(a,b)$ be in the interior of some strip $S_\Omega(\mathbf{v})$. Then, there exists a neighborhood of $(a,b)$ with finitely many winning saddle connections. 
\end{lemma}

\begin{proof}
Let $(a,b)$ be in the interior of the strip $S_\Omega(\mathbf{v})$ for $\mathbf{v} = (x, y)$ with $y > 0$ and $x \geq 0$. Then, we can find an $\epsilon > 0$ so that the closed ball of radius $\epsilon$ around $(a,b)$ remains in the interior of the strip. That is, we choose an $\epsilon > 0$ so that  $$\overline{B_\epsilon((a,b))} \subset S_\Omega(\mathbf{v}).$$

We can then use Lemma \ref{lem:finitewinners} to conclude that there are finitely many winning saddle connections on $\overline{B_\epsilon((a,b))}$ and therefore on $B_\epsilon((a,b))$.
\end{proof}

We now look at points $(a,b) \in T_X$ with $b > 0$ and show that these points have a neighborhood with finitely many winners. 

\begin{lemma}
\label{lem:b_pos}
Every point $(a,b) \in T_X$ such that $b > 0$ has a neighborhood $B_\epsilon((a,b))$ such that there are finitely many winning saddle connections on $B_\epsilon((a,b)) \cap T_X$. The same is also true for the point $(0,1/y_0)$. 
\end{lemma}

\begin{proof}
We recall that $T_X$ is a triangle bounded by the lines $b = \frac{-x_0}{y_0} a + \frac{1}{y_0}$ on top, the line $a=1$ on the right and the line $b =( \frac{-x_0}{y_0}-n) a + \frac{1}{y_0}$ on the bottom. 

We break up the proof of this lemma into cases, depending on the location of $(a,b) \in T_X \cup \{(0,1/y_0)\}$: 
\begin{enumerate}
    \item $0 < b < \frac{-x_0}{y_0} a + \frac{1}{y_0}$: These points are in the interior of $T_X$. We also notice that they must be in the interior of the strip $S_\Omega(\mathbf{v})$ for $\mathbf{v} = (x_0, y_0)$. Therefore, by Lemma \ref{lem:interiorstrip}, such a point $(a,b)$ must have a neighborhood with finitely many winning saddle connections. 
    \item $b= \frac{-x_0}{y_0} a + \frac{1}{y_0}$: These points are on the top line of $T_X$ but have $a > 0$. We recall that $y_0$ was chosen to be the least $y > 0$ for which $X$ has a saddle connection $(x,y_0)$. Then, $x_0$ was the least $x > 0$ for which $(x,y_0)$ was a saddle connection of $X$. 

    Let $(a,b)$ be any point on the top line of $T_X$ with $a > 0$ and let $\mathbf{v} = (x_0, y_0)$. Then $(a,b)$ is on the top line of the strip $S_\Omega((x_0,y_0))$. We can find an $\epsilon > 0$ such $\overline{B_\epsilon((a,b))} \cap S_\Omega((x_0,y_0))$ is a closed subset of $S_\Omega((x_0,y_0))$. By Lemma \ref{lem:finitewinners}, there are then finitely many winners on  $B_\epsilon((a,b)) \cap S_\Omega((x_0,y_0))$.  
    
    \item $(a,b) = (0,1/y_0)$: This point is not in $T_X$ but is the top left corner of the triangle that makes up $T_X$. 
    
    We can find a $y_1 > y_0$ such that every saddle connection $(x,y)$ of $X$ with $y > y_0$ must satisfy that $y \geq y_1$. Thus, we can then choose an $\epsilon > 0$ such $B_\epsilon((0,1/y_0)) \cap T_X \subset S_\Omega((x_0,y_0))$ and no strip $S_\Omega((x,y))$ for a saddle connection with $y > y_0$ and $x \geq 0$ intersects $B_\epsilon((0,1/y_0))$. This would imply that the only possible winning saddle connections on $B_\epsilon((0,1/y_0))$ are of the form $(x,y_0)$ for $x \geq x_0$. 
    
    But if we fix $y=y_0$, since the set of saddle connections $(x,y_0)$ is discrete and $T_X$ is bounded below by the line $b =( \frac{-x_0}{y_0}-n) a + \frac{1}{y_0}$, there are only finitely many saddle connections $\mathbf{v} = (x,y_0)$ of $(X, \omega)$ whose strip $S_\Omega(\mathbf{v})$ intersects $B_\epsilon((0,1/y_0))$ (exactly those $x$ such that $x_0 \leq x \leq x_0 + ny_0$). We have shown then that only finitely many strips $S_\Omega(\mathbf{v})$ for holonomy vectors $\mathbf{v}$ that could win over $(x_0, y_0)$ intersect $B_\epsilon((0,1/y_0))$ and therefore there are only finitely many winners on this neighborhood. 
\end{enumerate}
\end{proof}

Having established that points $(a,b) \in T_X$ with $b>0$ have neighborhoods with finitely many winners, we now turn to points $(a,b) \in T_X$ with $b\leq 0$. We first show that such points that are not on the right boundary of $T_X$ have finitely many winners. 

\begin{lemma}
\label{lem:b_neg}
Let $(a,b) \in T_X$ satisfy that $b \leq 0$ and $a < 1$. Then, there exists a neighborhood $B_\epsilon((a,b))$ such that there are finitely many winning saddle connections on $B_\epsilon((a,b)) \cap T_X$. 
\end{lemma}
\begin{proof}
By Lemma \ref{lem:interiorstrip}, it suffices to show that $(a,b)$ lies on the interior of a strip $S_\Omega(\mathbf{v})$ for some saddle connection $\mathbf{v}$.

 Because $(a,b)$ is in $T_X$, we know that it lies in some strip $S_\Omega(\mathbf{v})$. If $(a,b)$ is in the interior of $S_\Omega(\mathbf{v})$, then we are done. Otherwise if $(a,b)$ is on the boundary of $S_\Omega(\mathbf{v})$, we consider the points $p_\epsilon = ((1+\epsilon)a,(1+\epsilon)b)$, with winner $\mathbf{w}_\epsilon$. Since $(a,b)$ lies in the interior of $T_X$, for $\epsilon > 0$ sufficiently small, $p_\epsilon$ also lies in $T_X$. Moreover notice that $p_\epsilon$ and $(a,b)$ lie on the same line through the origin. This immediately implies that $(a,b)$ lies in the interior of $S_\Omega(\mathbf{w}_\epsilon)$, as seen in Figure \ref{fig:strip_wepsilon}.

\begin{figure}[ht]
\centering
\begin{tikzpicture}
\draw [fill = blue!25, blue!25] (0,0) -- (4,-3) -- (4,-1.5) -- (0,1.5) -- (0,0);
\node at (2,-.6) {$(a,b)$};
\node at (3,-1.1) {$p_\epsilon$};
\draw[thick, <->] (-1, 0) -- (4,0); 
\draw[thick, <->] (0,-3) -- (0,2); 
\draw[dashed] (0,0) -- (4,-3); 
\draw[fill] (2,-1) circle [radius = .05];
\draw[fill] (3,-1.5) circle [radius = .05];
\draw[dashed] (0,0) -- (3,-1.5);
\draw (0,1.5) -- (4,-1.5);
\end{tikzpicture}
\caption{The strip $\mathcal{S}_\Omega(\mathbf{w}_\epsilon)$. }
\label{fig:strip_wepsilon}
\end{figure}
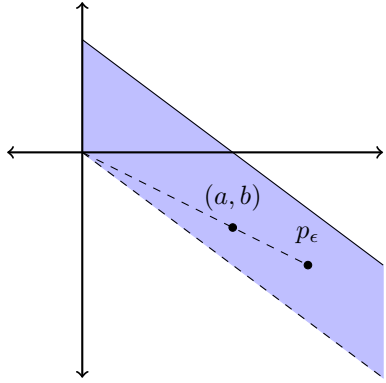

 Indeed, by the definition of $S_\Omega(\mathbf{v})$ for any holonomy vector $\mathbf{v}$ as a half-open strip with the open bottom boundary passing through the origin, for all points $p \in S_\Omega(\mathbf{v})$ the points $tp$ for $0 < t < 1$ lie in the interior, which gives the desired result.
\end{proof}

The combination of our previous lemmas shows that for all $(a,b) \in T_X$ away from the right vertical boundary, there are only finitely many winners in a neighborhood of $(a,b)$. We also want to show that for each $(1,b)$ on the right vertical boundary, there are only finitely many winners in a neighborhood. We will do this in two steps. First we will show that there are finitely many winning saddle connections along the right boundary of $T_X$. We will then use this result to prove that every point $(1,b)$ on the right boundary of $T_X$ has a neighborhood  with finitely many winning saddle connections. 

For our first result, we will need the following definition: 

\begin{defn}
\label{def:strip_lambda}
Given $(a,b) \in \mathbb{R}^2$, define the set $\mathcal{S}_\Lambda(a,b)$ as the strip of vectors $\mathbf{v} = (x,y) \in \mathbb{R}^2$ such that $0 < ax+by \leq 1$ and $y > 0$. This corresponds to the set of vectors that are potential winners on the surface $\mathbf{M}_{a,b}(X,\omega)$. 
\end{defn}

We think of this definition as a sort of dual to Definition \ref{def:strip_omega}, where instead of thinking of the surfaces corresponding to a particular winning saddle connection, we think about the the set of possible coordinates of winning saddle connections for a particular surface. 


\begin{lemma}
\label{lem:finiteboundary}
There are only finitely many winning saddle connections along the right vertical boundary $a=1$ of $T_X$.
\end{lemma}

\begin{proof}
By Lemma \ref{lem:winners_exist}, we know that every point $(1,b)$ on the right boundary of $T_X$ has a winning saddle connection. The set of $b \in \mathbb{R}$ such that $(1,b) \in T_X$ is some interval $[c,d]$. We note that since $\begin{bmatrix} 1 & \alpha \\ 0 & 1 \end{bmatrix}$ is in the Veech group of our surface for some $\alpha > 0$, it suffices to show that there are finitely many winners for $b \in [c + n\alpha, d + n\alpha]$ for any $n \in \mathbb{Z}$. This is because $(x,y)$ is the winner for $b'$ if and only if $(x-n\alpha y, y)$ is the winner for $b'+n\alpha$. For convenience, we will prove that there are finitely many winners for $b \in [M,N] = [c + n\alpha, d+n\alpha]$ for an $n$ such that $M,N > 0$. 

For each such $b$, we let $\mathbf{v}_b$ be its corresponding winning saddle connection. We wish to show that the set of vectors $\mathbf{v}_b$ is finite. We suppose that $\{\mathbf{v}_b\}$ is infinite. Then, since the set of $b \in [M,N]$, we must be able to find a convergent subsequence of $b_i \in \mathbb{R}$ such that $b_i \rightarrow b'$ and $b', b_i \in [M,N]$ for all $i$. In particular, $b' > 0$.

We claim now that $S_\Lambda(1,b')$ cannot have a winning saddle connection, which would contradict Lemma \ref{lem:winners_exist}. This corresponds to a saddle connection $(x,y)$ in the strip $S_\Lambda(1,b')$ that maximizes $\frac{x}{y}$. The strip $S_\Lambda(1,b')$ satisfies that $y > 0$ and $0 < x + b'y \leq 1$, or alternatively that $-\frac{1}{b'}x < y \leq  -\frac{1}{b'}x + \frac{1}{b'}$. We recall that $b' > 0$. Figure \ref{fig:strip_b'} shows a depiction of this strip.

\begin{figure}[ht]
\centering
\begin{tikzpicture}[scale = 1.1]
\draw [fill = blue!25, blue!25] (0,0) -- (1,0) -- (-2,3) -- (-2,2) -- (0,0); 
\draw[thick, <->] (-2, 0) -- (2,0); 
\draw[thick, <->] (0,-1) -- (0,3); 
\draw[dashed] (0,0) -- (-2,2);
\draw (1,0) -- (-2,3); 
\node at (1,-.3) {$1$}; 
\node at (.2,1.15) {$\frac{1}{b'}$};
\end{tikzpicture}
\caption{The strip $\mathcal{S}_\Lambda(1,b')$. }
\label{fig:strip_b'}
\end{figure}
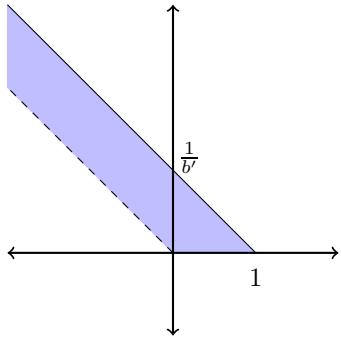

We suppose that the winning saddle connection $(x',y')$ for $b'$ lies in the interior of $S_\Lambda(1,b')$. If $\frac{x'}{y'} > \frac{x_i}{y_i}$ and $(x',y') \in S_\Lambda(1,b_i)$, then $(x_i,y_i)$ could not be the winner for $(1,b_i)$ because $(x',y')$ beats it and is still in the strip $S_\Lambda(1,b_i)$.

We let $\mathcal{C}_{b'}$ be the cone given by the intersection of $y < \frac{y'}{x'} x$ and $y > \frac{y'}{x'-1}x  - \frac{y'}{x'-1}$. We notice that if $(x_i, y_i) \in C_{b'}$, then it follows that $(x',y') \in S_\Lambda(1,b_i)$. One can see this algebraically or visually by noting that if $(x_i, y_i)$ is in the cone $\mathcal{C}_{b'}$ as depicted in Figure \ref{fig:strip_cone}, then $S_\Lambda(1,b_i)$ contains $(x_i, y_i)$ and is bounded by two lines with $x$-intercepts $0$ and $1$ and therefore must contain the point $(x',y')$. Furthermore, the first inequality defining the cone gives us that $\frac{x'}{y'} > \frac{x'}{y'}$.

Therefore, if $(x_i, y_i)$ is a winning saddle connection for some $(1,b_i)$ it cannot be in the open cone $\mathcal{C}_{b'}$ as defined above. Since $b_i \rightarrow b'$, this implies that for any $\epsilon > 0$ we can find an $n$ large enough such that the strips $S_\Lambda(1,b_i)$ all lie in a region $\mathcal{S}_n$ that is region where $(-\frac{1}{b'} + \epsilon)x \leq y \leq (-\frac{1}{b'} - \epsilon)x + (\frac{1}{b'} + \epsilon)$ and $y > 0$. Specifically, we will choose an $\epsilon$ such that the slopes of the two bounding lines of $\mathcal{S}_n$ are wedged between the slopes of the bounding lines of $\mathcal{C}_{b'}$. That is, we will choose $\epsilon > 0$ so that $(-\frac{1}{b'} - \epsilon) > \frac{y'}{x'}$ and $(-\frac{1}{b'} + \epsilon) < \frac{y'}{x'-1}$. We call this latter region $\mathcal{S}_n$. Figure \ref{fig:strip_cone} illustrates these regions. 


\begin{figure}[ht]
\centering
\begin{tikzpicture}[scale = 2]

\draw [fill = red, opacity = 0.2] (0,0) -- (1,0) -- (-1.75,3) -- (-2,3) -- (-2,1.75) -- (0,0); 

\draw[red] (0,0) -- (-2,1.75);
\draw[red] (1,0) -- (-1.75,3); 

\draw [fill = blue, blue, opacity = 0.2] (0,0) -- (1,0) -- (-2,3) -- (-2,2) -- (0,0); 
\draw[thick, <->] (-2, 0) -- (2,0); 
\draw[thick, <->] (0,-1) -- (0,3); 
\draw[dashed] (0,0) -- (-2,2);
\draw (1,0) -- (-2,3); 
\node at (1,-.3) {$1$}; 

\draw[fill] (-.25, .5) circle [radius = 0.025];
\node at (-.4, .8) {$(x',y')$}; 
\draw[orange, dashed] (0,0) -- (-1.5, 3); 
\draw[orange, dashed] (1,0) -- (-2,1.2);
\draw[fill = yellow, opacity = 0.2] (-1.5,3) -- (-2,3) -- (-2,1.2) -- (-.25,.5); 

\node[red] at (-.5,1.9) {$\mathcal{S}_n$}; 
\node[RedOrange] at (-1.6,.9) {$\mathcal{C}_{b'}$}; 
\node[blue] at (-1.5, 2) {$S_\Lambda(1,b')$};

\end{tikzpicture}
\caption{The strip $\mathcal{S}_\Lambda(1,b')$ with its winner $(x',y')$ and cone $\mathcal{C}_{b'}$, along with the region $\mathcal{S}_n$ containing the winners $(x_i, y_i)$ for $i \geq n$.}
\label{fig:strip_cone}
\end{figure}
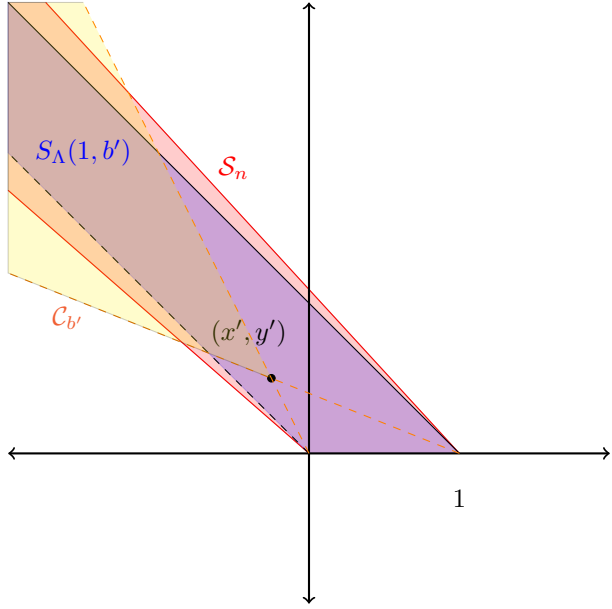

Given these conditions, we notice that $\mathcal{S}_n \backslash \mathcal{C}_{b'}$ is a compact set. With the possible exception of one point that equals $(x',y')$, the winning saddle connections $(x_i, y_i)$ for $i \geq n$ must all be in this region. But the set of holonomy vectors of saddle connections of $(X,\omega)$, of which $\{(x_i, y_i)\}$ is a discrete subset of $\mathbb{R}^2$ with no accumulation points, and so there are only finitely many $(x_i, y_i) \in \mathcal{S}_n \backslash \mathcal{C}_{b'}$. This contradicts the infiniteness of the set $\{(x_i, y_i)\}$. Hence, if $S_\Lambda ((1,b'))$ contained a point $(x',y')$, it could not be in the interior of the strip. 

We also consider the case when $(x', y')$ is in on the boundary of $S_\Lambda(1,b')$. That is, we suppose that $(x',y')$ is on the line $y = -\frac{1}{b'} x + \frac{1}{b}$. If there exists a saddle connection in the interior of $S_\Lambda(1,b')$, we can appeal to the reasoning in the previous case to find a contradiction. Else, Lemma \ref{lem:winners_exist} guarantees that there is also a holonomy vector $(x'', y'')$ on the open boundary $y = -\frac{1}{b'} x$ of $S_\Lambda(1,b')$. 

We now consider the cone $\mathcal{C}_{b'}'$ given by the intersection of the regions $y < \frac{y'}{x'}x$ and $y > \frac{y''}{x''-1}x  - \frac{y''}{x''-1}$. Similar to the previous case, we can find $n$ large enough such that the strips $S_\Lambda(1,b_i)$ all lie in a region $\mathcal{S}_n$ that is defined by $(-\frac{1}{b'} + \epsilon)x \leq y \leq (-\frac{1}{b'} - \epsilon)x + (\frac{1}{b'} + \epsilon)$ and $y > 0$. Here, we again choose $\epsilon > 0$ so that the slopes of the two bounding lines of $\mathcal{S}_n$ are wedged between the slopes of the bounding lines of $\mathcal{C}_{b'}'$. That is, we will choose $\epsilon > 0$ so that  $(-\frac{1}{b'} - \epsilon) > \frac{y'}{x'}$ and $(-\frac{1}{b'} + \epsilon) < \frac{y''}{x''-1}$. We call this latter region $\mathcal{S}_n$. Figure \ref{fig:strip_cone2} illustrates these regions.

\begin{figure}[ht]
\centering
\begin{tikzpicture}[scale = 2]

\draw [fill = red, opacity = 0.2] (0,0) -- (1,0) -- (-1.75,3) -- (-2,3) -- (-2,1.75) -- (0,0); 

\draw[red] (0,0) -- (-2,1.75);
\draw[red] (1,0) -- (-1.75,3); 

\draw [fill = blue, blue, opacity = 0.2] (0,0) -- (1,0) -- (-2,3) -- (-2,2) -- (0,0); 
\draw[thick, <->] (-2, 0) -- (2,0); 
\draw[thick, <->] (0,-1) -- (0,3); 
\draw[dashed] (0,0) -- (-2,2);
\draw (1,0) -- (-2,3); 
\node at (1,-.3) {$1$}; 

\draw[fill] (-1, 2) circle [radius = 0.025];
\node at (-.8, 2.2) {$(x',y')$}; 
\draw[fill] (-1, 1) circle [radius = 0.025];
\node at (-1.2, .8) {$(x'',y'')$}; 

\draw[orange, dashed] (0,0) -- (-1.5,3); 
\draw[orange, dashed] (1,0) -- (-2,1.5);
\draw[fill = yellow, opacity = 0.2] (-1.5,3) -- (-2,3) -- (-2,1.5) -- (-.33,.67); 

\node[red] at (-.4,1.8) {$\mathcal{S}_n$}; 
\node[RedOrange] at (-1.8,1.2) {$\mathcal{C}_{b'}$}; 
\node[blue] at (-1.5, 2) {$S_\Lambda(1,b')$};

\end{tikzpicture}
\caption{The strip $\mathcal{S}_\Lambda(1,b')$ with its winner $(x',y')$, the vector $(x'', y'')$ on its open boundary, its cone $\mathcal{C}_{b'}'$, along with the region $\mathcal{S}_n$ containing the winners $(x_i, y_i)$ for $i \geq n$.}
\label{fig:strip_cone2}
\end{figure}
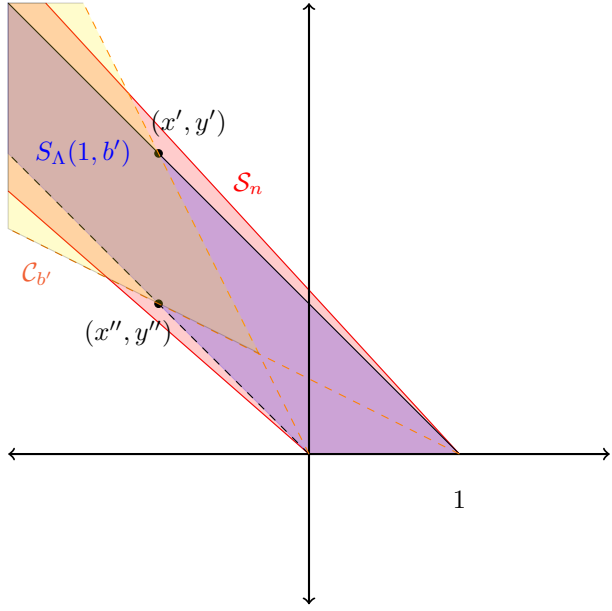

Since the set $\{(x_i, y_i)\}$ has no accumulation points and  $\mathcal{S}_n \backslash \mathcal{C}_{b'}'$ is compact, all but finitely many of the $\{(x_i,y_i)\}$ for $i \geq n$ must lie in the cone $\mathcal{C}_{b'}'$ and not be equal to $(x',y')$ or $(x'',y'')$. Let us consider one of these $(x_i,y_i)$. The corresponding strip $\mathcal{S}_\Lambda(1,b_i)$ is the region between two parallel lines that intersect the $x$-axis at $1$ and $0$, including the line through $1$ but not including the line through $0$. Therefore, $\mathcal{S}_\Lambda(1,b_i)$ must either contain $(x',y')$ or $(x'',y'')$, depending on if $b_i \leq b'$ or $b_i > b'$ respectively. If it contains $(x',y')$, then by similar reasoning as in the previous case $(x',y')$ beats $(x_i, y_i)$ and so $(x_i,y_i)$ could not have been the winner for $(1,b_i)$. If it contains $(x'',y'')$, then either $(x'',y'')$ beats $(x_i, y_i)$, which means that $(x_i,y_i)$ was not the winner, or $(x_i,y_i)$ was in the interior or $\mathcal{S}_\Lambda((1,b'))$, which contradicts that the interior of $S_\Lambda(1,b')$ did not contain any saddle connections. In either case, we have a contradiction.

Since we found a contradiction in both the cases when there was saddle connection in the interior and on the boundary of $S_\Lambda(1,b')$, we see that there must have been only finitely many winners on the right vertical boundary of $T_X$. 
\end{proof}

We can now use the previous lemma to show that points on the right boundary of $T_X$ have a neighborhood with finitely many winners. 

\begin{lemma}
\label{lem:boundarynbhd}
Given any point $(a,b) \in T_X$ with $a=1$, there exists a neighborhood $B_\epsilon((a,b))$ such that there are finitely many winning saddle connections on $B_\epsilon((a,b)) \cap T_X$. 
\end{lemma}

\begin{proof}
We suppose that we have a point $(a,b) \in T_X$ with $a=1$, $b=b'$. Then, Lemma \ref{lem:winners_exist} guarantees that $(1,b')$ is in some strip $S_\Omega(\mathbf{v})$. If $(1,b')$ is in the interior of $S_\Omega(\mathbf{v})$, then Lemma \ref{lem:interiorstrip} shows that there is a neighborhood of $(1,b')$ in $T_X$ with finitely many potential winners. 

We now consider the case where $(1,b')$ is not in the interior of any strip. This must then mean that $(1,b')$ is on the top boundary of some strip $S_\Omega(\mathbf{v})$. We will first deal with the case where $(1,b')$ is not on the top boundary of $T_X$. Then every point $(1,b'+c)$ for $c > 0$ small enough must also be in some winning strip. Since Lemma \ref{lem:finiteboundary} tells us that there are finitely many winning saddle connections on the right boundary of $T_X$ where $a=1$, this then implies that $(1,b')$ is on the bottom boundary of some other strip $S_\Omega(\mathbf{w})$, where $\mathbf{w}$ is the winning saddle connection for all $(1, b'+ c)$ for $c > 0$ small enough. 

Because there are finitely many winning saddle connections on the $a=1$ line of $T_X$ by Lemma \ref{lem:finiteboundary}, we can now choose an $\epsilon > 0$ small enough so that $\mathbf{w}$ is the winning saddle connection for $(1,b'+c)$ and $\mathbf{v}$ is the winning saddle connection for $(1,b'-c)$ for any $0 < c \leq \epsilon$.

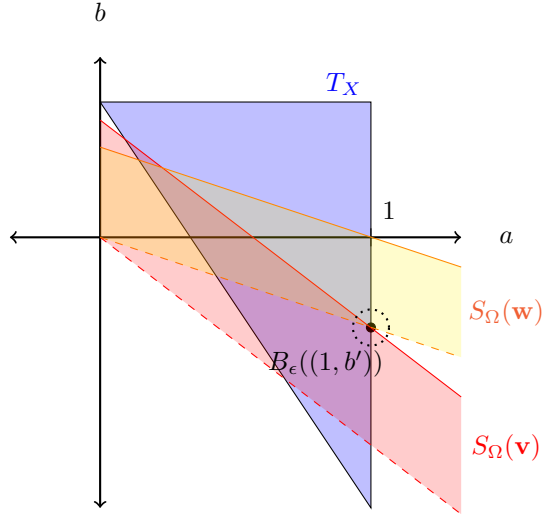
\begin{figure}[ht]
\centering
\begin{tikzpicture}[scale = 1.2]
\draw [fill = blue!25] (0,1.5) -- (3,1.5) -- (3,-3) -- (0,1.5);
\node at (4.5,0) {$a$};
\node at (0,2.5) {$b$}; 
\draw[thick, <->] (-1, 0) -- (4,0); 
\draw[thick, <->] (0,-3) -- (0,2); 
\draw (3,-.1) -- (3,.1); 
\node at (3.2,.3) {$1$}; 
\node[blue] at (2.7, 1.7) {$T_X$};

\draw[fill] (3,-1) circle [radius = 0.05]; 
\draw[red] (0,1.3) -- (4,-1.77);
\draw[red, dashed] (0,0) -- (4,-3.07);
\draw[orange] (0,1) -- (4,-.33);
\draw[orange, dashed] (0,0) -- (4,-1.33);
\fill[red, opacity = 0.2] (0,1.3) -- (4,-1.77) -- (4,-3.07) -- (0,0); 
\fill[yellow, opacity = 0.2] (0,1) -- (4,-.33) -- (4,-1.33) -- (0,0); 

\node at (2.5,-1.4) {$B_\epsilon((1,b'))$}; 
\node[red] at (4.5, -2.33) {$S_\Omega(\mathbf{v})$}; 
\node[RedOrange] at (4.5, -.83) {$S_\Omega(\mathbf{w})$}; 

\draw[thick, dotted] (3,-1) circle [radius = 0.2]; 
 
\end{tikzpicture}

\caption{The winning strips $S_\Omega(\mathbf{v})$ and $S_\Omega(\mathbf{w})$ near $(1,b')$ on the right boundary of $T_X$.}
\end{figure}

We claim now that there are finitely many winning saddle connections on $B_\epsilon((1,b'))$. We recall that for a point $(a,b) \in B_\epsilon((1,b')) \cap T_X$ to have a winning saddle connection other than $\mathbf{v}$ or $\mathbf{w}$, there must be a strip $S_\Omega(\mathbf{u})$ for a saddle connection $\mathbf{u}$ that is steeper (has more negative slope) than $S_\Omega(\mathbf{v})$ or $S_\Omega(\mathbf{w})$ (whichever is the winner at $(a,b)$) that contains $(a,b)$. 

Shrinking $\epsilon$ if necessary, $B_\epsilon((1,b'))$ lies above the line $b  = - \frac{x}{y}a + \frac{1}{h}$ for some $h > 0$ and $(x,y) = \mathbf{v}$. Then, as in the proof of Lemma \ref{lem:finitewinners}, we can show that there are finitely many strips saddle connections $\mathbf{u}$ of $(X,\omega)$ with strips $S_\Omega(\mathbf{u})$ intersect $B_\epsilon((1,b'))$ and that are at least as steep as $S_\Omega(\mathbf{v})$. 

If $S_\Omega(\mathbf{u})$ is at most as as steep as $S_\Omega(\mathbf{w})$, then it cannot win for any point in $B_\epsilon((1,b')) \cap T_X$ since $\mathbf{w}$ or $\mathbf{v}$ would win instead. 

If $S_\Omega(\mathbf{u})$ has steepness strictly between that of $\mathbf{w}$ and $\mathbf{v}$, then for $\mathbf{u}$ to be a winner for some point $(a,b) \in B_\epsilon((1,b')) \cap T_X$, we must have that $(a,b) \in S_\Omega(\mathbf{u}) \cap (S_\Omega(\mathbf{w}) \backslash S_\Omega(\mathbf{v}))$. But then, by slope considerations, $S_\Omega(\mathbf{u})$ must also intersect the $a=1$ boundary of $T_X$ in $B_\epsilon((1,b'))$ above the point $(1,b')$. But this contradicts that $\mathbf{w}$ and $\mathbf{v}$ were the only winners on the right boundary of $T_X$ in $B_\epsilon((1,b'))$. 

Hence, we have seen that only the finitely many saddle connections $\mathbf{u}$ with strips that intersect $B_\epsilon((1,b'))$ and have slope steeper than that of  $S_\Omega(\mathbf{v})$ can be winners on $B_\epsilon((1,b')) \cap T_X$. 

If $(1,b')$ were on the top boundary of $T_X$, then $(1,b')$ is on the top boundary of $S_\Omega(\mathbf{v})$ where $\mathbf{v} = (x_0, y_0)$. We can again reason as in the proof of Lemma \ref{lem:finitewinners} that we can find an $\epsilon > 0$ such that there are only finitely many saddle connections $\mathbf{u}$ with $S_\Omega(\mathbf{u})$ intersecting $B_\epsilon((1,b')) \cap T_X$ and with slope steeper than that of $\mathbf{v}$. This shows that there are finitely many winners on $B_\epsilon((1,b')) \cap T_X$.
\end{proof}

Combining these lemmas, this shows that for all points in $T_X$ there are finitely many winners in a neighborhood, and hence by compactness there are finitely many winners on $T_X$. 

\begin{proof}[Proof of Theorem \ref{thm:finite}] 
We will consider $\overline{T_X} = T_X \cup \{(0,1/y_0)\}$. This is a compact set. We showed in Lemmas \ref{lem:b_pos}, \ref{lem:b_neg}, and \ref{lem:boundarynbhd} that for any point $(a,b) \in \overline{T_X}$, we can find a neighborhood $B_\epsilon((a,b))$ such that there are finitely many possible winning saddle connections on $B_\epsilon((a,b)) \cap T_X$. Since $\overline{T_X}$ is compact, it is covered by finitely many of these neighborhoods. Since a finite union of finite sets is finite, the set of possible winners on $T_X$ is finite. 

Each winning saddle connection $\mathbf{v}_i$ would then be a winner on a convex polygonal piece of $T_X$. The cumulative distribution function of the slope gap distribution would then be given by the sums of areas between the level curves of the hyperbolic return time functions $\frac{y}{a(ax+by)}$ as described in Section \ref{subsec:algorithm} and the sides of these polygons. Since there are finitely many polygonal pieces, the cumulative distribution function and therefore also the slope gap distribution would be piecewise real analytic with finitely many points of non-analyticity.
\end{proof}

\section{Quadratic Tail Decay}\label{sec:decay} 

As an application of the finiteness result (Theorem \ref{thm:finite}), we prove
\tail

\begin{proof}

We find the decay of the tail on a piece of the Poincare section given by the triangle $T_X$. Doing this for all the pieces gives the decay of the tail. 

The proof of Theorem \ref{thm:finite} shows that there exists a minimal finite set of saddle connections $F\subset \Lambda(X,\omega)$ such that for any point in the triangle $T_X$, there is some $\mathbf{v} \in F$ with $\mathbf{M}_{a,b}\mathbf{v}$ being the winning saddle connection. Let  $S_\Omega(\mathbf{v})\subset T_X$ denote the strip where $\mathbf{M}_{a,b}\mathbf{v}$ could win and  $W_\Omega(\mathbf{v})\subset S_\Omega(\mathbf{v})$ denote where $\mathbf{M}_{a,b}\mathbf{v}$ does win.

Fix $\mathbf{v}=(x,y)\in F$. Then the tail on the piece $W_\Omega(\mathbf{v})$ is proportional to the area of the set points $(a,b)$ in $T_X$ with $$\text{Slope}(\mathbf{M}_{a,b}\mathbf{v}) = \frac{y}{a^2x+aby}>t\iff \frac{1}{at}- \frac{x}{y}a>b.$$
Let $m = \frac{x}{y}$. By adding the contribution that $W_\Omega(\mathbf{v})$  gives on the tail for each $\mathbf{v}\in F$ we get the full contribution to the tail. In what follows we work on one such winning saddle connection $\mathbf{v}$. Hence, it suffices to understand the portion of $W_\Omega(\mathbf{v})$ below the hyperbola $b=\frac{1}{at}-m a$. Notice that this hyperbola approaches the line $b=-ma$ from above. Moreover, notice that the line $b=-ma$ is the bottom boundary of the strip $S_\Omega(\mathbf{v})$.

 We have 3 situations depending on how the line $b=-ma$ intersects $T_X$, as shown in Figure \ref{fig:3cases}.

\begin{figure}[ht]
\centering
\begin{tikzpicture}
\begin{axis}[xmin=-0.5, xmax=1.1,
   		ymin=-1.25, ymax=0.75, axis lines = center, ticks = none, unit vector ratio=1 1 1]
    \addplot [domain=0:1, samples=100, name path=f, thick, color=blue!50]
        {0.5-0.75*x};
    
    \addplot [domain=0:1, samples=100, name path=g, thick, color=blue!50]
        {0.5-1.5*x};
        
    \addplot +[mark=none,color=blue!50,thick] coordinates {(1, -.25) (1,-1)};
    
    \addplot[blue!10, opacity=0.6] fill between[of=g and f, soft clip={domain=0:1}];
    
    \addplot [domain=0:1, samples=100, name path=h, thick, color=red!50]
        {-1.25*x};
        
    \addplot [domain=0:1, samples=100, name path=j, thick, color=red!50]
        {1/(8*x) -1.25*x};

    \addplot[red!50, opacity=0.8] fill between[of=g and f, soft clip={domain=0:0.207}];
    
    \addplot[red!50, opacity=0.8] fill between[of=g and j, soft clip={domain=0.207:0.293}];
\end{axis}
\node at (2,-.5) {Case (1)};
\end{tikzpicture}
\hspace{1cm}
\begin{tikzpicture}
\begin{axis}[xmin=-0.5, xmax=1.1,
   		ymin=-1.25, ymax=0.75, axis lines = center, ticks = none, unit vector ratio=1 1 1]
    \addplot [domain=0:1, samples=100, name path=f, thick, color=blue!50]
        {0.5-0.75*x};
    
    \addplot [domain=0:1, samples=100, name path=g, thick, color=blue!50]
        {0.5-1.5*x};
        
    \addplot +[mark=none,color=blue!50,thick] coordinates {(1, -.25) (1,-1)};
    
    \addplot[blue!10, opacity=0.6] fill between[of=g and f, soft clip={domain=0:1}];
    
    \addplot [domain=0:1.25, samples=100, name path=h, thick, color=red!50]
        {-1*x};
        
    \addplot [domain=0:1.25, samples=100, name path=j, thick, color=red!50]
        {1/(10*x) -1*x};

    \addplot[red!50, opacity=0.8] fill between[of=g and f, soft clip={domain=0:0.183}];
    
    \addplot[red!50, opacity=0.8] fill between[of=g and j, soft clip={domain=0.183:0.276}];

    \addplot[red!50, opacity=0.8] fill between[of=g and j, soft clip={domain=0.724:1}];
\end{axis}
\node at (2,-.5) {Case (2)};
\end{tikzpicture}
\hspace{1cm}
\begin{tikzpicture}
\begin{axis}[xmin=-0.5, xmax=1.1,
   		ymin=-1.25, ymax=0.75, axis lines = center, ticks = none, unit vector ratio=1 1 1]
    \addplot [domain=0:1, samples=100, name path=f, thick, color=blue!50]
        {0.5-0.75*x};
    
    \addplot [domain=0:1, samples=100, name path=g, thick, color=blue!50]
        {0.5-1.5*x};
        
    \addplot +[mark=none,color=blue!50,thick] coordinates {(1, -.25) (1,-1)};
    
    \addplot[blue!10, opacity=0.6] fill between[of=g and f, soft clip={domain=0:1}];
    
    \addplot [domain=0:1.25, samples=100, name path=h, thick, color=red!50]
        {-.75*x};
        
    \addplot [domain=0:1.25, samples=100, name path=j, thick, color=red!50]
        {1/(16*x) -.75*x};

    

\end{axis}
\node at (2,-.5) {Case (3)};
\end{tikzpicture}
\caption{An illustration of the three cases in the proof of Theorem \ref{thm:tail}.}
\label{fig:3cases}
\end{figure}
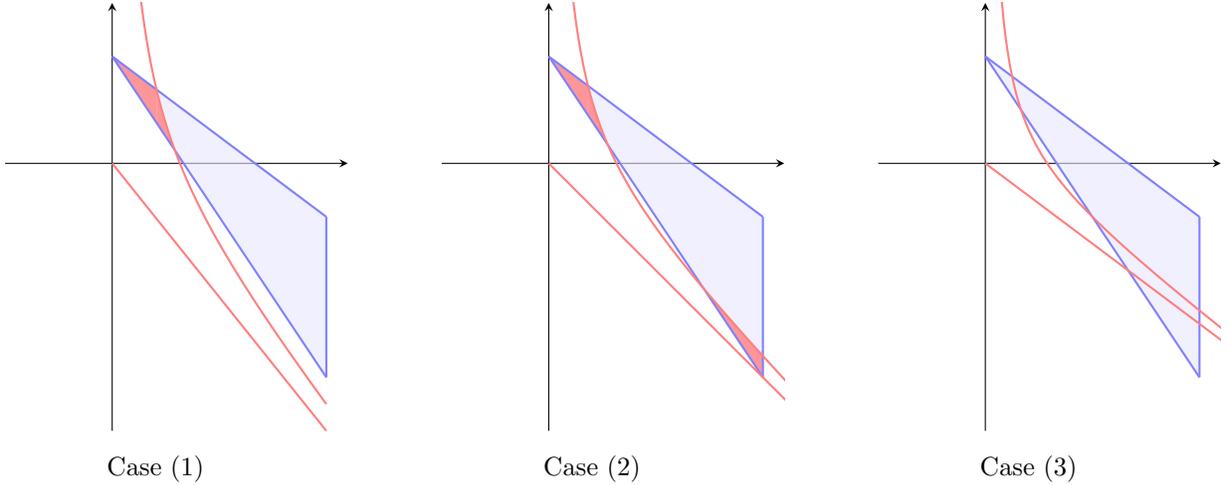

\begin{enumerate}
\item Suppose  $b=-ma$  doesn't intersect $T_X$.  This means that the slope of the line $b=-ma$ is less than the slope of the bottom edge of $T_X$. In this case we will only find contribution to the tail when the vertical of $\mathbf{v}$ is $y_0$ since otherwise we can choose large enough $t$ so that the hyperbola misses $W_\Omega(\mathbf{v})$. 

An upper bound for the contribution of $W_\Omega(\mathbf{v})$ is just the part underneath the hyperbola and inside the $T_X$. For $t$ large, the hyperbola $b= \frac{1}{at}- ma$ intersects the triangle twice. First it intersects at the top through the boundary line $b = \frac{1}{y_0}-\frac{x_0}{y_0}a$ at the point

 $$a_{top} ^+ = \frac{-1+\sqrt{1+\frac{4y_0(my_0-x_0)}{t}}}{2(m y_0-x_0)}$$

and then leaves through the bottom boundary line $b = \frac{1}{y_0}-(\frac{x_0}{y_0}+n)a$ at the point 

$$a_{bot} ^+ = \frac{-1+\sqrt{1+\frac{4y_0(my_0-(x_0+ny_0))}{t}}}{2(m y_0-(x_0+ny_0))}.$$

Thus, the contribution is given by 
$$\int_{a=0} ^{a_{top} ^+} \int_{b= \frac{1}{y_0}-(\frac{x_0}{y_0}+n)a} ^{\frac{1}{y_0}-\frac{x_0}{y_0}a}1\,db\,da+
\int_{a=a_{top} ^+} ^{a_{bot} ^+} \int_{b= \frac{1}{y_0}-(\frac{x_0}{y_0}+n)a} ^{\frac{1}{at}- ma}1\,db\,da.$$

The first integral evaluates to $\frac{n}{2}(a_{top} ^+)^2$ and, by using a Taylor series on the square root, can be shown to decay like $t^{-2}$.

The second integral evaluates to 
$$\left(\frac{1}{t}\log(a)+\frac{1}{2}\left(\frac{x_0}{y_0}+n-m\right)a^2-\frac{1}{y_0}a\right) \bigg|_{a=a^+ _{top}} ^{a^+ _{bot}}.$$
By preforming a Taylor series approximation on $a_{top} ^+$ and $a_{bot} ^+$ we get that the second integral gives an decay like $t^{-3}$.

Thus, the total decay on the integral is like $t^{-2}$.

\item Now consider the case when $b=-ma$ intersects $T_X$ at the bottom vertex of $T_X$. In this case we have $m=\frac{x_0}{y_0}+n-\frac{1}{y_0}.$ If the vertical of $y$ is the same as $y_0$, then we get a contribution to the tail at the top of $T_X$ as in case 1. In fact, this is the only way we can get contribution at the top of $T_X$.

 Now we find the contribution on the bottom of $T_X$. Thus, we are interested in the intersection of the hyperbola $b=\frac{1}{at}-ma$ with the bottom boundary line of $T_X$ given by $\frac{1}{y_0}-(\frac{x_0}{y_0}+n)a$. This is the point
$$a_{bot} ^- = \frac{-1-\sqrt{1+\frac{4y_0(my_0-(x_0+ny_0))}{t}}}{2(m y_0-(x_0+ny_0))}.$$
In fact, using that the line $b=-ma$ intersects the bottom of $T_X$, we get that $m=\frac{x_0}{y_0}+n-\frac{1}{y_0}$ and so we can see 
$$a_{bot} ^- = \frac{1}{2}\left(1+\sqrt{1-\frac{4y_0}{t}}\right).$$
 The contribution is then given by
$$\int_{a=a^- _{bot} } ^1  \int_{b= \frac{1}{y_0}-(\frac{x_0}{y_0}+n)a} ^{\frac{1}{at}- ma}1\,db\,da.$$
This integral evaluates to
$$\frac{1}{2}\left(\frac{x_0}{y_0}+n-m\right)-\frac{1}{y_0}-\frac{1}{t}\log(a^- _{bot})-\frac{1}{2}\left(\frac{x_0}{y_0}+n-m\right)(a^- _{bot})^2+\frac{a^- _{bot}}{y_0}.$$
By doing a Taylor series approximation on $a^- _{bot}$ we can show that the decay is like $t^{-2}$.

\item 
  
  Now suppose that line $b=-ma$ does intersect $T_X$ and this intersection is above the bottom vertex of $T_X$ i.e. above $b=\frac{1}{y_0}-(\frac{x_0}{y_0}+n)$. Then each point on the bottom edge of $S_\Omega(\mathbf{v})$ must be in some other winning strip $S_\Omega(\mathbf{v}')$. There are finitely many such $\mathbf{v}'$ that we will number $\mathbf{v}_1,\ldots, \mathbf{v}_n$. Thus, $W_\Omega(\mathbf{v}) \subset \left(S_\Omega(\mathbf{v}) - \bigcup_{i=1}^n S_\Omega(\mathbf{v}_i)\right)$, which is some polygonal region whose closure is completely above the bottom boundary of $S_\Omega(\mathbf{v})$, $b = -ma$. Since the hyperbola $b = \frac{1}{at} - ma$ approaches $b=-ma$ as $t \rightarrow \infty$, we see that for all $t$ large enough, the hyperbola is completely below $W_\Omega(\mathbf{v})$ and therefore $W_\Omega(\mathbf{v})$ has no contribution to the tail.

Adding up the contribution of every $\mathbf{v}\in F$, we see that there is a constant $C>0$, such that
$$\int_t ^\infty f(x)\,dx\le \frac{C}{t^2}.$$
\end{enumerate}
Now we compute a lower bound. Let $\mathbf{v}_0 = (x_0,y_0)$ be the saddle connection used to define $T_X$, $S_\Omega(\mathbf{v}_0)$ denote the associated strip, and $b = \frac{1}{at}-\frac{x_0}{y_0}a$ be the associated hyperbola. We will use this specific saddle connection to find a lower bound to $\int_t ^\infty f(x)\,dx$ essentially by using the argument from case 1 of the upper bound. That is, by analzying the behavior at the top of the triangle. Either $\mathbf{v}_0$ is the winning saddle connection for every point on $S_\Omega(\mathbf{v}_0)$ or there is some other saddle connection $\mathbf{v}$ for which it is the winning saddle connecton on $S_\Omega(\mathbf{v}_0)\cap S_\Omega(\mathbf{v})$. We deal with both cases.

\begin{enumerate}
\item If $\mathbf{v}_0$ is the winning saddle connection for every point on $S_\Omega(\mathbf{v}_0)$, then a lower bound to $\int_t ^\infty f(x)\,dx$ comes from the part underneath the hyperbola $b = \frac{1}{at}-\frac{x_0}{y_0}a$ and inside $S_\Omega(\mathbf{v}_0)$. We can choose $t$ large enough so that the hyperbola intersects $S_\Omega(\mathbf{v}_0)$ only once at the point
 $$a_{top} ^+ = \frac{-1+\sqrt{1+\frac{4y_0(my_0-x_0)}{t}}}{2(m y_0-x_0)}$$
with contribution given by
$$\int_{a=0} ^{a_{top} ^+} \int_{b= \frac{1}{y_0}-(\frac{x_0}{y_0}+n)a} ^{\frac{1}{y_0}-\frac{x_0}{y_0}a}1\,db\,da.$$

Earlier we showed this decays like $t^{-2}$.

\item In the case that there is some other saddle connection $\mathbf{v}$ for which it is the winning saddle connection on $S_\Omega(\mathbf{v}_0)\cap S_\Omega(\mathbf{v})$ we have two subcases depending on whether $\mathbf{v}$ has the same vertical as $\mathbf{v}_0$ or not. In the latter case we can choose $t$ large enough so that the contribution is the same as case 1 of the previous case. We now focus on when the vertical of $\mathbf{v}$ and $\mathbf{v}_0$ is the same. The contribution is given by
$$\int_{a=0} ^{a_{top} ^+} \int_{b=\frac{1}{y_0}-\frac{x}{y}a} ^{\frac{1}{y_0}-\frac{x_0}{y_0}a}1\,db\,da.$$
The integral evaluates to $\frac{1}{2}\left(\frac{x}{y}-\frac{x_0}{y_0}\right)(a_{top} ^+)^2$ and, by using a Taylor series on the square root, can be shown to decay like $t^{-2}$.
\end{enumerate}
\end{proof}
\section{Further Questions} 
\label{sec:questions}
We end with a few questions for further exploration. 

\begin{enumerate}
\label{sec:conclusion}
\item Are there bounds on the number of points of non-analyticity of the slope gap distribution of a Veech surface? 

 In \cite{BMMUW}, linear upper and lower bounds in terms of $n$ on the number of points was found for the translation surface given by gluing opposite sides of the $2n$-gon. These surfaces each have two cusps and have genus that grows linearly in $n$. This shows that bounds on the number of points of non-analyticity based on the number of cusps is impossible. However, we can ask if there are bounds based on the genus of the surface.



\item What can be said about the gap distributions of non-Veech surfaces?

In \cite{AC1}, it was shown that the limiting slope gap distribution exists for almost every translation surface, and in \cite{S}, the slope gap distributions for a special family of non-Veech surfaces were shown to be piecewise real-analytic. We can ask if the limiting slope gap distributions are always piecewise real-analytic, and if so, are there always finitely many points of non-analyticity.

\item Where do the points of non-analyticity lie?

Beyond just understanding the number of points of non-analyticity, we can ask about number-theoretic properties of the points themselves. In every example known to the authors of a limiting slope gap distribution, after rescaling, the points of non-analyticity lie in the trace field of the Veech group. Given that the gap distribution is computed by integrating areas between hyperbolas in regions related to the geometry of the surface, it is natural to conjecture that points of non-analyticity lie in quadratic extensions of the trace field.

\end{enumerate}



\bibliographystyle{alpha}
\bibliography{paper}

\newcommand{\etalchar}[1]{$^{#1}$}
\begin{thebibliography}{BMMM{\etalchar{+}}21}

\bibitem[AC12]{AC1}
Jayadev~S. Athreya and Jon Chaika.
\newblock The distribution of gaps for saddle connection directions.
\newblock {\em Geom. Funct. Anal.}, 22(6):1491--1516, 2012.

\bibitem[AC14]{AC2}
Jayadev~S. Athreya and Yitwah Cheung.
\newblock A {P}oincar\'{e} section for the horocycle flow on the space of
  lattices.
\newblock {\em Int. Math. Res. Not. IMRN}, (10):2643--2690, 2014.

\bibitem[ACL15]{ACL}
Jayadev~S. Athreya, Jon Chaika, and Samuel Leli\`evre.
\newblock The gap distribution of slopes on the golden {L}.
\newblock In {\em Recent trends in ergodic theory and dynamical systems},
  volume 631 of {\em Contemp. Math.}, pages 47--62. Amer. Math. Soc.,
  Providence, RI, 2015.

\bibitem[Ath16]{A}
Jayadev~S. Athreya.
\newblock Gap distributions and homogeneous dynamics.
\newblock In {\em Geometry, topology, and dynamics in negative curvature},
  volume 425 of {\em London Math. Soc. Lecture Note Ser.}, pages 1--31.
  Cambridge Univ. Press, Cambridge, 2016.

\bibitem[BMMM{\etalchar{+}}21]{BMMUW}
Jonah Berman, Taylor Mcadam, Ananth Miller-Murthy, Caglar Uyanik, and Hamilton
  Wan.
\newblock Slope gap distribution of saddle connections on the 2n-gon.
\newblock arXiv:2109.04495, 2021.

\bibitem[HS06]{HS}
Pascal Hubert and Thomas~A. Schmidt.
\newblock An introduction to {V}eech surfaces.
\newblock In {\em Handbook of dynamical systems. {V}ol. 1{B}}, pages 501--526.
  Elsevier B. V., Amsterdam, 2006.

\bibitem[Mas88]{M}
Howard Masur.
\newblock Lower bounds for the number of saddle connections and closed
  trajectories of a quadratic differential.
\newblock In {\em Holomorphic functions and moduli, {V}ol. {I} ({B}erkeley,
  {CA}, 1986)}, volume~10 of {\em Math. Sci. Res. Inst. Publ.}, pages 215--228.
  Springer, New York, 1988.

\bibitem[Mas90]{M2}
Howard Masur.
\newblock The growth rate of trajectories of a quadratic differential.
\newblock {\em Ergodic Theory Dynam. Systems}, 10(1):151--176, 1990.

\bibitem[San21]{S}
Anthony Sanchez.
\newblock Gaps of saddle connection directions for some branched covers of
  tori.
\newblock {\em Ergodic Theory and Dynamical Systems}, pages 1--55, 2021.

\bibitem[Tah19]{Taha}
Diaaeldin Taha.
\newblock On cross sections to the geodesic and horocycle flows on quotients of
  ${SL}(2,\mathbb{R}) $ by {Hecke} triangle groups ${G}_q$.
\newblock {\em arXiv preprint arXiv:1906.07250}, 2019.

\bibitem[UW16]{UW}
Caglar Uyanik and Grace Work.
\newblock The distribution of gaps for saddle connections on the octagon.
\newblock {\em Int. Math. Res. Not. IMRN}, (18):5569--5602, 2016.

\bibitem[Vor05]{V}
Yaroslav Vorobets.
\newblock Periodic geodesics on generic translation surfaces.
\newblock In {\em Algebraic and topological dynamics}, volume 385 of {\em
  Contemp. Math.}, pages 205--258. Amer. Math. Soc., Providence, RI, 2005.

\bibitem[Zor06]{Z}
Anton Zorich.
\newblock Flat surfaces.
\newblock In {\em Frontiers in number theory, physics, and geometry. {I}},
  pages 437--583. Springer, Berlin, 2006.

\end{thebibliography}

\end{document}